\documentclass[12pt, final ]{amsart}
\usepackage{enumerate}
\usepackage{amssymb,amsfonts,amsmath,amsthm}
\usepackage{tikz}
\usepackage{showkeys}
\usepackage[all]{xy}

\usetikzlibrary{arrows,calc,through,backgrounds,matrix,decorations.pathmorphing}
\tikzset{node distance=2cm, auto}

\theoremstyle{definition}
\newtheorem{example}{Example}[section]
\newtheorem{remark}[example]{Remark}
\newtheorem{assumption}[example]{Assumption}
\newtheorem{definition}[example]{Definition}

\theoremstyle{plain}
\newtheorem{proposition}[example]{Proposition}
\newtheorem{corollary}[example]{Corollary}

\newtheorem{theorem}[example]{Theorem}
\newtheorem{lemma}[example]{Lemma}

\newcommand{\NN}{\mathbb{N}}
\newcommand{\ZZ}{\mathbb{Z}}

\newcommand{\sA}{\mathcal{A}}
\newcommand{\sB}{\mathcal{B}}
\newcommand{\sM}{\mathcal{M}}
\newcommand{\sN}{\mathcal{N}}

\newcommand{\sO}{\mathcal{O}}

\newcommand{\sF}{{\mathcal{F}}}

\newcommand{\sK}{\mathcal{K}}
\newcommand{\sL}{\mathcal{L}}

\newcommand{\bA}{{\bf A}}
\newcommand{\bB}{{\bf B}}

\newcommand{\vr}{{\vec{r}}}
\newcommand{\vs}{{\vec{s}}}

\newcommand{\fY}{{\mathfrak{Y}}}

\newcommand{\isoarrow}{\stackrel{\sim}{\rightarrow}}

\newcommand{\Div}{{\mathfrak Div}}

\DeclareMathOperator{\Spec}{Spec}

\DeclareMathOperator{\coh}{\mathfrak Coh}

\DeclareMathOperator{\Hom}{Hom}

\DeclareMathOperator{\spec}{Spec}

\DeclareMathOperator{\func}{Func}
\DeclareMathOperator{\bker}{{\bf ker}}
\DeclareMathOperator{\face}{Face}
\DeclareMathOperator{\supp}{supp}

\DeclareMathOperator{\tr}{tr}
\let\div\relax
\DeclareMathOperator{\div}{div}
\DeclareMathOperator{\lcm}{lcm}

\newcommand{\Gm}{\mathbb{G}_{\mathrm{m}}}

\begin{document}
\title{$G$-theory of root stacks and equivariant $K$-theory}

\author[A. Dhillon]{Ajneet Dhillon}

\address{Department of Mathematics, Middlesex College, University of
Western Ontario, London, ON N6A 5B7, Canada}

\email{adhill3@uwo.ca}

\author[I. Kobyzev]{Ivan Kobyzev}

\address{Department of Pure Mathematics, University of
Waterloo, Waterloo, ON N2L 3G1, Canada}

\email{ikobyzev@uwaterloo.ca}

\begin{abstract}{Using the description of the category of quasi-coherent sheaves on a root
stack, we compute the $G$-theory of root stacks via localisation methods.
We apply our results to the study of equivariant $K$-theory of algebraic varieties 
under certain conditions.} \end{abstract}

\maketitle
\section{Introduction}
Let $X$ be an algebraic variety equipped with an action of a finite group $G$. One would like to compute the equivariant $K$-theory $K_G(X)$. A first answer was given in the paper \cite{duke} in the case when $X$ is a smooth curve. Let us briefly describe it. We set $Y $ to be the quotient $X/G$, and $\phi: X \rightarrow Y$ the quotient map, and $B$ the branch locus. Then $B$ is a finite union of $G$ orbits $B_1, \ \dots, \ B_n$. Choosing a point $P_i \in B_i$ for each $i$, denote the inertia group of $P_i$ by $H_i$. Note that it is a cyclic group. Using some basic properties of equivariant sheaves and the Borel construction it was proved that there is a decomposition of abelian groups:
$$K_G(X) = K(Y) \oplus \oplus_{i=1}^n R_k'(H_i) , $$ 
where  $R_k'(H)$ is the subgroup of a representation ring without invariants, i.e. $x \in R_k(H)$, such that $\langle x, 1_H \rangle  = 0$. From here we can guess a flavor of the result in general case: there should be some kind of a decomposition of $K_G(X)$ onto $K(Y)$ and the  terms coming from ramification.

If one wants to generalize this to higher dimensions, there are two routes one may take.
One may enter the realm of algebraic stacks.  For example, Vistoli and Vezzosi in the  papers \cite{vistoli} and \cite{vv} proved the decomposition formula for $K_G(X)$ of a scheme $X$ using (implicitly) a top down description of the stack $[X/G]$.  

Another route would be to enter the realm of logarithmic geometry,
see \cite{niziol} and \cite{hagihara}. These two papers study the $K$-theory of the 
Kummer \'{e}tale site on a logarithmic scheme. It should be noted that, using the correspondence between sheaves on an infinite root stack and sheaves on the Kummer \'{e}tale site
(see \cite[\S 6]{tv}), one can deduce the structure results of \cite[\S 4]{hagihara} and \cite[Theorem 1.1]{niziol} from our Theorem \ref{th:G-root} and Corollary \ref{c:kequalsg}.

We begin by discussing the general philosophy of our approach encompassing both of these routes.
In algebraic geometry one frequently needs to consider equivariant objects on a scheme $X$ with
respect to the action of $G$. These objects correspond to objects over the quotient stack
$[X/G]$. However, it can happen that $[X/G]\cong [X'/G']$ for seemingly unrelated
$X$ and $X'$. In such situation, it is useful to have a canonical description of the 
quotient stack $[X/G]$, perhaps in terms of its coarse moduli space $Y$. This may
not always be possible but sometimes it is. In this paper we will describe a situation
in which this occurs, see Theorem \ref{main theorem}. When our hypotheses are satisfied, the
quotient stack becomes a root stack over its coarse moduli space $Y$.

The root stack construction goes back to Olsson (see \cite{O}). If a quotient stack is ``a tool" to take quotients, similarly a root stack can be used to ``extract roots" from line bundles on a scheme. It turns out that this construction is quite useful, for example, in Gromov-Witten theory of a Deligne-Mumford stack (see \cite{AGV}, \cite{C},  \cite{O}).   The moduli stack of stable maps from a curve to a stack does not have nice properties, and instead one need to consider so called twisted stable maps from a twisted curve. As was shown in \cite{AGV}, one can replace a twisted curve by a root stack. 

Another application of  root stacks is the parabolic orbifold correspondence.
In a nutshell, this correspondence describes sheaves and vector bundles on a root stack 
in  terms of sheaves and vector bundles on the base with extra data.
 Parabolic bundles on a Riemann surface were defined in \cite{MS}, and were shown to be related to a unitary representation of a homotopy group. Borne proved the equivalence of parabolic bundles and locally free sheaves on a root stack (see \cite{borne}). Finally, Borne and Vistoli generalized it to the equivalence of  quasi-coherent sheaves on a root stack and parabolic sheaves (see \cite{BV}).

The results of \cite{BV} are the foundation of this work. Using their description of coherent sheaves on a root stack, we are able to compute the algebraic $G$-theory of a root stack. See Theorem \ref{th:G-root} for the statement of our first main result. The tool necessary for its proof is localisation sequences associated with a quotient category. The reader can think about this method as an algebraic analog of Segal's localization theorem for equivariant topological $K$-theory (see \cite[Prop 4.1]{segal}).

The second result of this work is Theorem \ref{main theorem}. It says that under certain assumptions a quotient stack is a root stack over its coarse moduli space. The main tool used in the proof is a generalization of Abhyankar's lemma, see \cite[XIII, appendix I]{SGA1}.

Combining  these results gives an immediate application to equivariant $K$-theory of  schemes. This is how we obtain a generalization of the aforementioned decomposition of \cite{duke}.   We formulated it as Theorem \ref{K-theorem}. If a finite group $G$ acts on a scheme $X$, then, under some assumptions, we have the decomposition of
$K_G(X)$ into direct sum of groups $ K(X/G)  $ and $G$-theory of ramification divisors and their intersections. Note that our assumptions will always be satisfied for tame actions of groups on smooth projective curves.

Let us give an outline of the paper for the convenience of the reader. 
In a short preliminary section \ref{s:local} we recall some necessary categorical techniques. 

We start by studying the $G$-theory of a root stack in section~\ref{s:root}. First, the description of the category of 
quasi-coherent sheaves on a root stack by \cite{BV} in section \ref{subsection-BV} is recalled. After that 
 we exploit localisation methods to decompose  the $G$-theory of parabolic sheaves. Finally, in  section \ref{ss:computation} we combine all intermediate results and formulate Theorem \ref{th:G-root}, giving the $G$-theory of a root 
stack over a noetherian scheme. We finish the section with a remark that under some assumption, the algebraic $G$-theory of a root stack coincides with its Waldhausen $K$-theory in the sense of \cite{Joshua}, Corollary \ref{c:kequalsg}.

In section \ref{s:quotient} we address the issue of when a quotient stack is a root stack. First we show that under our assumptions (tameness of the action and ramification divisor is normal crossing), the inertia group is generated in codimension one (see Theorem \ref{t:inertia}). We use Abhyankar's theorem (\cite[Th 2.3.2]{GrM}) in the proof. Then under the same hypothesis, we show that a quotient stack is a root stack, see Theorem \ref{main theorem}.

The paper ends in section \ref{s:equi}, where we study equivariant $K$-theory of a scheme by combining the results of the previous two sections.
  As an example we compute the equivariant $K$-theory of the affine line and the Burniat surface. \\

\textit{Acknowledgments.} We are grateful to N. Borne, J. Jardine and R. Joshua  for their interest in the project and useful remarks. We thank M. Satriano for the great discussions; after reading his paper \cite{GS}  we were able to remove an extra hypothesis in the first version of this work.  In addition we thank the referee for their helpful corrections.

\section*{Notations and conventions}
\begin{tabular}{cl}
$k$ & our base field \\
$\bker$ & the kernel of a functor (Def \ref{d:kernel}) \\
$\vr$  & an $n$-tuple $(r_1,\ldots, r_n)$ of real numbers \\
$\vr I^n$ & the poset of integer points in $\prod_{i=1}^n [0,r_i]$\\ 
$\func(\mathbf{A}, \mathbf{B})$ & functor category between two abelian categories \\
$\widehat{M}$ & the dual $\Hom(M,\Gm)$ of the monoid $M$ \\ 
$\Div X$ & the symmetric monoidal category of line bundles with section, see \S 3.1 \\
$X_{L,\vr}$ & a stack of roots over the scheme $X$, see definition 3.4 \\
$\coh X$ & category of coherent sheaves on $X$\\
$\mathcal{EP}(X,L,\vr)$ & category of coherent extendable pares (Rem \ref{CohExt})  \\
$S_n(k)=S(k)$ & The set of subsets of $\{1,2,\ldots, n\}$ of cardinality $k$\\
              & We often drop the subscript $n$ when it is clear from the context.\\
$\face^k$     & The $k$th face functor, see definition \ref{d:face} \\
$\bker^k$     & The kernel of the face functor, see definition \ref{d:ker}\\

\end{tabular}


\section{Localization via Serre subcategories}\label{s:local}

\subsection{Serre subcategories}

Let $\bA$ be an abelian category. Recall that a \emph{Serre subcategory} $S$ of $\bA$ is a non-empty full subcategory that is closed
under extensions, subobjects and quotients. When $\bA$ is well-powered the quotient category $\bA/S$ exists, see \cite[pg. 44, Theorem 2.1]{swan}.

We will need the following result to identify quotient categories.

\begin{theorem}
Let $F:\bA\rightarrow \bB$ be an exact functor between abelian categories. Denote
by $S$ the full subcategory whose objects are $x$ with $F(x)\cong 0$. 
Then $S$ is a Serre subcategory and we have a factorisation 
\begin{center}
\begin{tikzpicture}
\node (TL) at (0,0) {$\bA$};
\node (BL) [below of =TL] {$\bB$};
\node (TR) [right of=TL] {$\bA/S$};
\draw [->] (TL) edge node[auto] {$F$} (BL) ;
\draw [->] (TL) -- (TR);
\draw [->] (TR) -- (BL);
\end{tikzpicture}
\end{center}
\end{theorem}

\begin{proof}
See \cite[page 114]{swan}
\end{proof}

\begin{definition}\label{d:kernel}
The category $S$ is called the \textit{kernel of the functor $F$} and is denoted by $\bker(F).$
\end{definition}

\begin{theorem}\label{t:swanQuot}
In the situation of the previous theorem suppose that we have
\begin{enumerate}
\item for every object  $y\in \bB$ there is a $x\in \bA$ such that $F(x)$ is 
isomorphic to $y$ and
\item for every morphism $f:F(x) \rightarrow F(x')$ there is $x''\in \bA$ 
with $h:x''\rightarrow x$ and
$g:x''\rightarrow x'$ such that $F(h)$ is an isomorphism and the following diagram
commutes 
\begin{center}
\begin{tikzpicture}
\node (TL) at (0,0) {$F(x'')$};
\node (BL) [below of =TL] {$F(x)$};
\node (BR) [right of=BL] {$F(x')$.};
\draw [->] (TL) edge node[auto,left] {$F(h)$} (BL) ;
\draw [->] (TL) edge node[auto] {$F(g)$} (BR);
\draw [->] (BL) edge node[auto] {$f$} (BR);
\end{tikzpicture}
\end{center}
Then there is an equivalence of categories $\bA/S\cong \bB$.
\end{enumerate} 
\end{theorem}
\begin{proof}
See \cite[pg. 114, theorem 5.11]{swan}.
\end{proof}

\subsection{Some functor categories}

Consider  $n$-tuples of  integers $\vr=(r_1,r_2,\ldots, r_n)$
and $\vs=(s_1,s_2,\ldots, s_n)$. We denote by $[\vr,\vs]$ the poset
of $n$-tuples $(x_1,\ldots, x_n)$ with 
\[
x_i\in\ZZ\quad\text{and}\quad r_i\le x_i\le s_i.
\]
We will make use of the following shorthand notation :
\[
rI=[0,r] \quad\text{and}\quad \vr I^n=[0,\vr].
\]
These intervals are naturally posets with
\[
(x_1,x_2,\ldots, x_n)\le (y_1,y_2,\ldots, y_n) \quad\text{if and only if}
\quad x_i\le y_i\ \text{for all }i. 
\]
This poset structure allows us to view them as categories in the usual way.

Fix an abelian category $\bA$ and consider the functor category
 $$\func(\vr I^n,\bA).$$  This category is   
abelian  with kernels and cokernels formed pointwise.
We will be interested in the
$K$-theory of such categories. In this subsection we will try to understand some of their
quotient categories.
Given an object $\sF$ in this category and an object
$u$ of $\vr I^n$ we denote by $\sF_{u}\in\bA$ the value of the functor
$\sF$ on this object and if  $u\le v$ the arrow from $F_u$ to $F_v$ will be denoted 
by
\[
F_{+(v-u)}:F_u\rightarrow F_v. 
\]
In particular, we take $e_i=(0,0,\ldots,1,0,\ldots, 0)$ to be a standard basis vector
so that we have a morphism
\[
F_{+e_i}:F_{(u_1,\ldots, u_n)}\rightarrow F_{u_1,\ldots, u_{i-1},u_i+1,u_{i+1}\ldots, u_n}.
\]
\begin{lemma}\label{l:data}
To give an object $F$ of $\func(\vr I^n,\bA)$ is the same as providing the following
data :
\begin{enumerate}[(D1)]
\item objects $F_{(u_1,u_2,\ldots, u_n)}\in \bA$
\item arrows 
\[ F_{+e_i}:F_u \rightarrow F_{u+e_i},\]

\end{enumerate}
such that all diagrams of the form
\begin{center}
\begin{tikzpicture}
\node (TL) at (0,0) 
{ $F_{u}$};
\node (TR) at (5,0)
{ $F_{u+e_j}$};
\node (BL) at (0,-2)
{ $F_{u+e_i}$};
\node (BR) at (5,-2)
{ $F_{u+e_i+e_j}$};
\draw[->] (TL) -- (TR);
\draw[->] (TL) --(BL);
\draw[->] (BL) -- (BR);
\draw[->] (TR) -- (BR);
\end{tikzpicture}
\end{center}
\end{lemma}
\begin{proof}
The hypothesis insure that if $u\le v$ in $\vr I^n$ then
there is a well defined map $F_u\rightarrow F_v$ which produces
our functor.
\end{proof}

\begin{proposition}
\begin{enumerate}[(i)]
\item 
Let $\tr_{n-1}(\vec{r}) = (r_1,r_2,\ldots, r_{n-1})$.
There is an exact functor
\[
\pi : \func(\vr I^n ,\bA) \rightarrow \func(\tr_{n-1}(\vr)I^{n-1},\bA)
\]
defined on objects by 
$$
\pi(G)_{(u_1,u_2,\ldots, u_{n-1})}= (G)_{(u_1,\ldots, u_{n-1},0)}$$

\item The functor $\pi$ has  a left adjoint denoted $\pi^*$. 
We have $\pi \circ \pi^* \simeq 1$.

\item The functor $\pi^*$ is fully faithful.

\end{enumerate}
\end{proposition}

\begin{proof}
\begin{enumerate} [(i)]
\item There is an inclusion functor $\tr_{n-1}(\vr)I^{n-1}\hookrightarrow \vr I^n$ defined by
\[
(x_1,x_2,\ldots, x_{n-1})\mapsto (x_1,x_2,\ldots, x_{n-1},0).\]
The functor $\pi$ is just the restriction along this inclusion. The 
exactness follows from the fact that  in functor categories, limits and
colimits are computed pointwise.

\item 
Given $F\in \func(\tr_{n-1}(\vr)I^{n-1},\bA)$ we need to construct 
 an object $\pi^*(F)\in \func( \vr I^{n},\bA)$. We set
\[
\pi^*(F)_{(u_1,u_2,\ldots, u_n)} = F_{(u_1,u_2,\ldots, u_{n-1})}.
\]
To produce a functor, we need maps
\[
\lambda^i_{(u_1,\ldots, u_n)}:\pi^*(\sF)_{(u_1,\ldots, u_i,\ldots, u_n)}\rightarrow 
\pi^*(F)_{(u_1,\ldots, u_i+1,\ldots, u_n)}  
\]
We define
\[
 \lambda^i_{(u_1,\ldots, u_n)}=\begin{cases}
F_{(u_1,\ldots, u_i,\ldots, u_{n-1})}\rightarrow F_{(u_1,\ldots, u_i+1,\ldots, u_{n-1})} & \text{ if }i<n \\
\text{identity} &\text{ if } i=n.
\end{cases}
\]
One checks that the hypothesis of Lemma \ref{l:data} are satisfied. Observe that
$\pi \circ \pi^* =1$. This produces a natural map
\[
\Hom(\pi^*(F),G)\rightarrow \Hom(F, \pi(G)).
\]

To see that this is a bijection, suppose that we are given a morphism
$\beta : F\rightarrow \pi (G)$. There is a diagram, where
the dashed arrow is defined to be the composition,
\begin{center}
\begin{tikzpicture}
\node (TL) at (0,0) {$ \pi^*(F)_{(u_1,\ldots, u_n)} $};
\node (TR) at (4,0) {$ G_{(u_1,\ldots, u_n)} $};
\node (BL) at (0,-2) {$ F_{(u_1,\ldots, u_{n-1})} $};
\node (BR) at (4,-2) {$ G_{(u_1,\ldots, u_{n-1},0)} $};
\draw [dashed,->] (TL) -- (TR);
\draw [->] (BL) -- node [auto, below] {$\beta$} (BR);
\draw [double equal sign distance] (TL) -- (BL);
\draw [->] (BR) --  (TR);
\end{tikzpicture}
\end{center}
This produces a natural morphism
\[
\Hom(\pi^*(F),G)\leftarrow \Hom(F, \pi(G))
\]
and we check that it is inverse to the previous map.

\item We have
\[
\Hom(\pi^*(F), \pi^*(F')) = \Hom(F, \pi\pi^*(F')) = \Hom(F,F').
\] 
\end{enumerate}
\end{proof}

\begin{theorem}
\label{standard}
\begin{enumerate}
\item 
The functor 
\[
\pi : \func(\vr I^n,\bA) \rightarrow \func(\tr_{n-1}(\vr) I^{n-1},\bA)
\]
satisfies the hypothesis of Theorem \ref{t:swanQuot}. 
\item 
Let $\vs = (r_1,r_2,\ldots, r_{n-1}, r_n-1)$.
 If $r_n>0$ then
the kernel of this functor is equivalent to $\func(\vs I^n,\bA)$.
\item
If $r_n=0$ then there is an equivalence of categories
$$\func(\vr I^n,\bA)\cong \func(\tr_{n-1}(\vr)I^{n-1},\bA).$$ 
\end{enumerate}
\end{theorem}

\begin{proof}
(1) 
 The functor $\pi$ is exact so it remains to check the two conditions
 of the theorem. The first condition follows from the fact that
 $\pi \circ \pi^* $ is the identity. Now suppose that we have a
 morphism $\pi (F) \rightarrow \pi(F')$. By adjointness we obtain a diagram
 \begin{center}
 \begin{tikzpicture}
 \node (T) at (0,0) {$\pi^* \pi (F)$} ;
 \node (B) at (0,-1.5) {$F$} ;
 \node (BR) at (2,-1.5) {$F '$} ;
 \draw[->] (T) -- (B);
 \draw[->] (T) -- (BR);
 \end{tikzpicture}
 \end{center}
 Applying $\pi$ to this picture shows that the second condition holds.

(2)
   The functor $\pi$ was defined on objects by the rule $\pi(G)_{(u_1,u_2,\ldots, u_{n-1})}= (G)_{(u_1,\ldots, u_{n-1},0)}.$ So it is clear that if  $\pi G \cong 0$ then $(G)_{(u_1,\ldots, u_{n-1},0)} \cong 0$ and to give an object $G$ of ${\bker \pi}$ is the same (up to isomorphism) as  giving the objects  $(G)_{(u_1,\ldots, u_n)} \in \bA  $ for all $u \in \vr I^n, u_n \neq 0.$ And according to Lemma \ref{l:data} it is the same as providing an object of the category 
   $\func(\vs I^n,\bA)$
   
 (3) If $r_n=0$ then we have an equivalence of categories 
 $\tr_{n-1}(\vr)\cong \vr$.

     \end{proof}


\section{Coherent sheaves on root stacks} \label{s:root}
\subsection{Preliminary results}
\label{subsection-BV}
Recall that if $M$ is a commutative monoid then $\widehat{M}=\Hom(M,{\mathbb G}_m)$ is its dual.

In this subsection  we will recall the main constructions and theorems from
\cite{BV}. We refer the reader to this paper for further details. Let's start by defining a root stack.
  
Let $X$ be a scheme. Denote by $\mathfrak{Div}X$ the groupoid of line bundles over $X$  with sections. It has the structure of a symmetric monoidal category with  tensor product given by 
$$(L, s) \otimes (L', s') = (L\otimes L', s \otimes s') .$$

Choosing  $n$ objects $(L_1, s_1), ... (L_n, s_n)$ of  $\mathfrak{Div}X$ allows us
to define a symmetric monoidal functor (see \cite[Definition 2.1]{BV})
\begin{align*}
L : \mathbb{N}^n &\longrightarrow \mathfrak{Div}X \\
(k_1, ..., k_n) &\longmapsto (L_1,s_1)^{\otimes k_1}\otimes ... \otimes (L_n,s_n)^{\otimes k_n} .
\end{align*}

Such functors arise from morphisms $X\rightarrow [\spec \ZZ[\NN^n]/\widehat{\NN^n}]$.
Let us recall how.

\begin{proposition} \label{p:equiv}
\begin{enumerate}[(i)]
\item Let $\bA$ be the groupoid whose objects are quasi-coherent $\sO_X$-algebras $\sA$ with 
a $\ZZ^n=\widehat{\widehat{\NN^n}}$-grading $\sA=\oplus_{u\in \ZZ^n} \sA_u$ such that each summand
$\sA_u$ is an invertible sheaf. The morphisms are graded algebra isomorphisms.
Then there is an equivalence of categories between $\bA^{\rm op}$ and the groupoid of 
$\widehat{\NN^n}$-torsors $P\rightarrow X$.

\item Let $\bB$ be the groupoid whose objects are pairs $(\sA, \alpha)$
where $\sA$ is a sheaf of algebra satisfying the conditions in (i) and 
\[\alpha : \sO_X[\NN^n]\rightarrow \sA \] 
is a morphism respecting the grading. The morphisms in the category $\bB$ are 
graded algebra morphisms commuting with the structure maps. Then there is an equivalence
of categories between $\bB^{\rm op}$ and the groupoid of morphisms
 $X\rightarrow [\spec \ZZ[\NN^n]/\widehat{\NN^n}]$
\end{enumerate}
\end{proposition}

\begin{proof} This proposition is a summary of the discussion in \cite[p. 1343-1344]{BV}, in particular the proof of Proposition 3.25. The detailed proof can be found there. 
Here we will just illustrate the main idea behind the proof.

(i) The torsor $\pi: P\rightarrow X$ is determined by the sheaf of algebras $\pi_*(\sO_P)$
which has  a $\widehat{\NN^n}$-action and hence a weight grading. As the torsor
is locally trivial, the condition about the summands being invertible follows by
considering the algebra associated with the trivial torsor.

(ii) This follows from the standard description of the groupoid of $X$-points of a 
quotient stack. Finally, in \cite{BV}, the $\textbf{fppf}$ topology is needed but in 
the present work is  not. The setting in \textit{loc. cit.} is more general and
the monoids in question may have torsion, so that the torsor $P$ is a torsor over
$\mu_n$.  Such a torsor may not be trivial in the Zariski topology, unlike a ${\mathbb G}_m$-torsor. Hence a finer topology is needed. See the proof of \cite[Lemma 3.26]{BV}.
\end{proof}

\begin{corollary}
\label{Xpoints}
There is an equivalence of categories between the groupoid of symmetric monoidal 
functors $\NN^n\rightarrow \Div X$ and the groupoid of $X$-points of 
$[\spec \ZZ[\NN^n]/\widehat{\NN^n}]$.
\end{corollary}

\begin{proof} For details see \cite[Proposition 3.25]{BV}. 
In essence, the symmetric monoidal functor determined by $(L_1, s_1), \dots, (L_n, s_n)$
produces the graded sheaf of algebras 
\[
\sA = \bigoplus_{\vec{u}\in \ZZ^n} L_1^{u_1}\otimes \ldots \otimes L_n^{u_n}. 
\]
The sections produce an algebra map
\[
\sO_X[\NN^n]\rightarrow \sA.
\]
\end{proof}

\begin{definition} 
  Let $\vr=(r_1, r_2, ..., r_n)$ be a collection of positive natural numbers. 
  We denote by $r_i\NN$ the monoid $\{ vr_i | v\in \NN\}$. 
  We denote by $\vr \NN^n$ the monoid 
  \[
  \vr \NN^n = r_1\NN \times r_2\NN \times \ldots r_n\NN.
  \]
  We will view our symmetric monoidal functor above as a functor
  \[
  L:\vr\NN^n\rightarrow \Div X
  \]
  in the following way :
  \[
  (r_1\alpha_1, r_2\alpha_2,\ldots, r_n\alpha_n) \longmapsto
  (L_1,s_1)^{\otimes \alpha_1}\otimes \ldots (L_n, s_n)^{\otimes \alpha_n}.
  \]
  Consider the natural inclusion of monoids $j_{\vr}:\vr\mathbb{N}^n\hookrightarrow  \mathbb{N}^n$. 
   The \emph{category of $\vr$th roots of $L$} denoted by $(L)_\vr$, is defined as follows :
  
   Its objects are pairs $(M, \alpha)$, where $M:   \mathbb{N}^n \rightarrow  \mathfrak{Div}X$ is a symmetric monoidal functor, and $\alpha: L \rightarrow M \circ j $ is an isomorphism of symmetric monoidal functors. 
   
   An arrow from $(M, \alpha)$ to $(M', \alpha')$ is an isomorphism h$: M \rightarrow M'$ of symmetric monoidal functors $ \mathbb{N}^n \rightarrow  \mathfrak{Div}X$, such that the diagram 
   \begin{center}
\begin{tikzpicture}
\matrix (m) [matrix of math nodes, row sep=3em, column sep = 3em]
{& L \\
M \circ j &   & M'\circ j\\};
\draw [->] (m-1-2) to node[above] {$\alpha$} (m-2-1);
\draw [->] (m-1-2) to node[above] {$\alpha'$} (m-2-3);
\draw[->] (m-2-1) to node{$h\circ j$} (m-2-3);
\end{tikzpicture}
\end{center}
   commutes.
\end{definition} 
 
 This category is in fact a groupoid as a morphism $\phi$ in $\mathfrak{Div}X$, whose tensor power
 $\phi^{\otimes k}$ is an isomorphism, must be an isomorphism to begin with.
  
  Given a morphism of schemes $t: T\rightarrow X$ there is pullback functor
  $$ t^* : \Div X\rightarrow \Div T . $$
  Hence we can form the category of roots $(t^* \circ L)_\vr $. This construction pastes 
  together to produce a pseudo-functor
  \[
  ({\rm Sch}/X)\rightarrow \Div_X,
  \]
  where $\Div_X\rightarrow {\rm Sch}/X$ is the symmetric monoidal stack described
  in \cite[p. 1335]{BV}.
  
\begin{definition} In the above situation, the fibered category associated with this 
pseudo-functor is called 
\emph{the stack of roots} associated with $L$ and $\vr$. It is denoted by $X_{L,\vr}$. 
\end{definition}

We will often denote the stack of roots by
\[
X_{L,\vr} = X_{(L_1,s_1,r_1),\ldots ,(L_n,s_n,r_n) }.
\]

There are also two equivalent definitions of the stack $X_{L,\vr}$ and the equivalence is proved in \cite[Proposition 4.13]{BV} and \cite[Remark 4.14]{BV}. Lets recall the
description of this stack as a fibered product.
\begin{proposition}
The stack $X_{L,\vr}$ is isomorphic to the fibered product
$$ X\times_{\spec \ZZ[\vec{r}\NN^n]} [\spec \ZZ[\NN^n]/ \widehat{\NN^n}] $$
\end{proposition}

 According to (a slightly modified version of) Corollary \ref{Xpoints} a symmetric monoidal functor   $L: \vr \NN^n \rightarrow \mathfrak{Div}X$ corresponds to a morphism $X \rightarrow [\spec \ZZ[\vr \NN^n]/\widehat{\vr \NN^n}]$ which in turn corresponds to a $\widehat{\vr \NN^n} $-torsor $\pi : P \rightarrow X $ and a $\widehat{\vr \NN^n} $-equivariant morphism $P \rightarrow \spec \ZZ[\vr \NN^n].$  This gives

\begin{proposition}
\label{p:quotient}
The stack $X_{L,\vr}$ is isomorphic to the quotient stack
$$ [P \times_{\spec \ZZ[\vec{r}\NN^n]} \spec \ZZ[\NN^n]  / \widehat{\NN^n}], $$
where the action on the first factor is defined through the dual of the inclusion $j_{\vr} : \vr  \NN^n \hookrightarrow \NN^n. $

\end{proposition}

\begin{proof}
See \cite[p.1350]{BV}.
\end{proof}

We recall the definition of parabolic sheaf, see \cite[Definition 5.6]{BV}.

\begin{definition}
\label{parabolic}
 Consider a scheme $X$, an inclusion $\vr\ZZ^n \subseteq \ZZ^n$ and a symmetric monoidal functor
 $L:\vr\ZZ^n \rightarrow \Div_X$, defined by
 $$
 L_u=L(u)=L_1^{\alpha_1} \otimes \ldots L_n^{\alpha_n},
 $$
 where $u=(r_1\alpha_1,\ldots, r_n\alpha_n)$, and each $\alpha_i \in \ZZ$. 
 A \emph{parabolic sheaf} $(E, \rho) $ on $(X,L)$ with denominators $\vr$ consists 
 of the following data:
 \begin{enumerate}[(a)]
 \item A functor $E: \mathbb{Z}^n \rightarrow \mathfrak{QCoh}X,$ denoted by $v \mapsto E_v$ on objects and $b \mapsto E_b $ on arrows. 
 \item For any $u \in \vec{r} \mathbb{Z}^n $ and $v \in \mathbb{Z}^n ,$ an isomorphism of $ \mathcal{O}_X$-modules:
 $$ \rho^E_{u,v}: E_{u+v} \simeq L_u \otimes_{\mathcal{O}_X}E_v.$$
    This map is called the \textit{pseudo-period isomorphism}.  
   \end{enumerate}  
    This data are required to satisfy the following conditions. Let $u, u' \in \vec{r} \mathbb{Z}^n ,$ $a =(r_1\alpha_1,\ldots, r_n\alpha_n)\in \vec{r} \mathbb{N}^n, b \in  \mathbb{N}^n, v \in  \mathbb{Z}^n .$ Then the following diagrams commute.
    
    \begin{enumerate}[(i)]
    
\item $\phantom{text}$
   \begin{center}
\begin{tikzpicture}
  \node (TL) {$E_v$};
  \node (BL) [below of=TL] {$\mathcal{O}_X \otimes E_v$};
  \node (TR) [node distance=5cm,right of=TL] {$E_{a +v}$};
  \node (BR) [node distance=5cm,right of=BL] {$ L_a \otimes E_v $};
  \draw[->] (TL) to node {$E_a$} (TR);
  \draw[->] (TL) to node {$\simeq$} (BL);
  \draw[->] (BL) to node {$\sigma_a^L\otimes id_{E_v}$} (BR);
  \draw[->] (TR) to node {$\rho^E_{a,v}$} (BR); ,
\end{tikzpicture}
\end{center}
  where $\sigma_a=\sigma^{\alpha_1}\otimes\ldots\sigma^{\alpha_n}\in {\rm H}^0(X,L_a)$

   \item $\phantom{text}$
   \begin{center}
\begin{tikzpicture}
  \node (TL) {$E_{u+v}$};
  \node (BL) [below of=TL] {$E_{u+b+v}$};
  \node (TR) [node distance=5cm,right of=TL] {$ L_u\otimes E_v$};
  \node (BR) [node distance=5cm,right of=BL] {$ L_u\otimes E_{b+v} $};
  \draw[->] (TL) to node {$\rho^E_{u,v}$} (TR);
  \draw[->] (TL) to node {$E_b$} (BL);
  \draw[->] (BL) to node {$\rho^E_{u,b+v}$} (BR);
  \draw[->] (TR) to node {$id\otimes E_b$} (BR); ,
\end{tikzpicture}
\end{center}

$\phantom{text}$

  \item $\phantom{text}$
   \begin{center}
\begin{tikzpicture}
  \node (TL) {$E_{u+u'+v}$};
  \node (BL) [below of=TL] {$L(u) \otimes E_{u'+v}$};
  \node (TR) [node distance=6cm,right of=TL] {$L_{u+u'} \otimes  E_v$};
  \node (BR) [node distance=6cm,right of=BL] {$ L_u \otimes L_{u'}  \otimes E_{v} $};
  \draw[->] (TL) to node {$\rho^E_{u+u',v}$} (TR);
  \draw[->] (TL) to node {$\rho^E_{u, u'+v}$} (BL);
  \draw[->] (BL) to node {$id \otimes \rho^E_{u',v}$} (BR);
  \draw[->] (TR) to node {$\mu \otimes id$} (BR); ,
\end{tikzpicture}
\end{center}

\item The map
 \begin{center}
   \begin{tikzpicture}
 \matrix (m) [matrix of math nodes, row sep=3em, column sep = 4em]
{E_v = E_{0+v} & \mathcal{O}_X \otimes E_v \\ };
\draw [->] (m-1-1) to node[above] {$\rho^E_{0,v}$} (m-1-2);
    \end{tikzpicture}
  \end{center} 
 is the natural isomorphism.  
    \end{enumerate}
\end{definition}

\begin{definition}
A parabolic sheaf $(E, \rho) $ is said to be \textit{coherent} if for each $v \in \ZZ^n$ the sheaf $E_v$ is a coherent sheaf on $X.$ 
\end{definition}

\begin{theorem}[Borne, Vistoli ]
\label{t:parabolic}
Let $X$ be a scheme and $L$ is a monoidal functor  defined as in the beginning of the subsection. Then there is a canonical tensor equivalence of abelian categories between the category $\mathfrak{QCoh} X_{L,\vr} $ and the category of parabolic sheaves on $X$, associated with $L$. 
\end{theorem}

\begin{proof} See \cite[Proposition 5.10, Theorem 6.1]{BV} for details.
The proof relies on the description of the stack as a quotient as in Proposition \ref{p:quotient}. 
From this description, sheaves on the stack are equivariant sheaves on 
$$
P\times_{\spec \ZZ[\vr \NN^n]} \times \spec \ZZ[\NN^n].
$$
As remarked in the proof of Proposition \ref{p:equiv}, the torsor $P$ is obtained from
a sheaf of algebras on $X$. The sheaf of algebras $\sA$ is constructed from 
the functor $L$ by taking a direct sum construction, it has a natural grading.
It follows that the scheme 
\[
P\times_{\spec \ZZ[\vr \NN^n]} \spec \ZZ[\NN^n] =
\spec(\sA\otimes_{\ZZ[\vr\NN^n]} \ZZ[\NN^n]).
\]
The algebra on the right has a natural $\ZZ[\NN^n]$ grading, see the corollary below for a local description. It follows that the
equivariant sheaves on the scheme in question are just graded modules over this algebra.
The proof follows be reinterpreting the graded modules in terms of the symmetric monoidal
functor $L$.
\end{proof}
  Actually we can add the finiteness condition to the previous theorem and get the following

\begin{corollary}
\label{equivalence}
Let $X$ be locally noetherian scheme.
There is a canonical tensor equivalence of abelian categories between the category $\mathfrak{Coh} X_{L,\vr} $ and the category of coherent parabolic sheaves on $X,$ associated with $L.$ 
\end{corollary}

\begin{proof}
  We will make use of the identifications in the above proof. Further the question is local on
  $X$, so we may assume that $X$ is in fact an affine scheme $\spec(R)$. By further restrictions we can assume that all the line bundles $L_i$ are in fact trivial, and
  we identify them with $R$. 
  In this
  situation the symmetric monoidal functor corresponds to a graded homomorphism
  \[
  \ZZ[X_1,X_2,\ldots, X_n]\rightarrow R[t_1^{\pm 1},t_2^{\pm 1},\ldots, t_n^{\pm 1}]
  \]
  sending $X_i$ to $x_it_i$ with $x_i\in R$. Further the morphism 
  $$\spec(\ZZ[\NN^n]) \rightarrow \spec \ZZ[\vec{r}\NN^n]$$ comes from an integral extension of
  algebras
  \[
  \ZZ[X_1,X_2,\ldots, X_n][Y_1,\ldots, Y_n]/(Y_1^{r_1}-X_1,\ldots Y_n^{r_n}-X_n)
  \]
  Then taking tensor products yields a $\ZZ^n$-graded algebra
  \[
 A=  R[t_1^{\pm 1},t_2^{\pm 1},\ldots, t_n^{\pm 1}][s_1,\ldots ,s_n]/(s_1^{r_1}-x_1t_1, \ldots, 
  s_n^{r_n}-x_nt_n)
  \]
  where $s_i$ has degree $(0,\ldots,0, 1,0\ldots, 0)=e_i$.  Consider now $M$ a finitely generated graded $A$-module.
  We can assume that the generators of $M$ are in fact homogeneous and hence there is an
  epimorphism 
  \[
  \oplus_{i=1}^p A(n_i) \rightarrow M.
  \] 
  The graded pieces of the module on the left are free of rank $p$ and hence 
  the graded pieces of $M$ are finitely generated. It follows that
  a finitely generated $A$-module gives rise to a parabolic sheaf with 
  values in the category of finitely generated $R$-modules, in other words
  coherent sheaves on $X$.
  
  Conversely suppose that each we have a graded $A$-module $M$ with each graded piece a finitely 
  generated $R$-module. We can find finitely many elements of $M$, let's say
  $\{ \alpha_1, \alpha_2, \ldots , \alpha_p\} $ of degrees 
  \[
  \deg (\alpha_i) = (\lambda_{i1},\lambda_{i2},\ldots, \lambda_{in})\in \ZZ^n
  \]
  with $0\le \lambda_{ij} \le r_j$ such that the associated morphism
  \[
  \phi:\oplus_{i=1}^p A(\deg(\alpha_i)) \rightarrow M 
  \]
  is an epimorphism in degrees 
  \[
  (\mu_1,\mu_2,\ldots , \mu_n)\in \ZZ^n
  \]
  whenever $0\le \mu_i \le r_i$. It follows that $\phi$ is an epimorphism at 
  multiplication by $t_i$ induces an isomorphism $M_{v}\isoarrow M_{v+e_i}$.
 \end{proof}

\subsection{An extension lemma}
The goal of this subsection is to slightly simplify the formulation of parabolic sheaf
in the present context using the pseudo-periodicity condition. This will be needed
to study $K$-theory in the next section. We let 
\[
e_i=(0,\ldots, 0, 1, 0, \ldots, 0)\in \NN^n,
\]
where the $1$ is in the $i$th spot.

\begin{definition}
	\label{d:EP}
Let $X$ be a scheme, and $L$ be a symmetric monoidal functor 
\[
L:\vr\ZZ^n\rightarrow \Div_X,
\]
determined by $n$ divisors $(L_i, s_i)$. 
An \emph{extendable pair} $(F,\rho)$ on $(X,L)$ consists of the following data: 
\begin{enumerate}[(a)]
 \item A functor
 $F_\bullet:   \vr I^n \rightarrow \mathfrak{QCoh}(X).$  
 
\item   For any $\alpha \in  \vr I^n  $ such that  $\alpha_i=r_i,$ an isomorphism of $\sO_X$-modules 

$$
\rho_{\alpha,\alpha-r_ie_i}:F_\alpha
\stackrel{\sim }{\rightarrow}  L_i \otimes F_{\alpha-r_ie_i}.
$$

We will frequently drop the subscripts from the notation involving $\rho$, when
they are clear from the context.
\end{enumerate}

This data is required to satisfy the following three conditions :
\begin{enumerate}[(EX1)]

\item \label{ex1} For all $i \in \{1, \ldots , n\}$  and $\alpha \in \vr I^n $   the following diagram commutes
\begin{center}
\begin{tikzpicture}
\matrix (m) [matrix of math nodes, row sep=4em, column sep=8em]
{F_\alpha & F_{\alpha+ (r_i-\alpha_i)e_i}  \\
   L_i \otimes F_{\alpha}& L_i \otimes F_{\alpha - \alpha_i  e_i } \\};
\draw[->] (m-1-1) to node {$F_{+(r_i - \alpha_i) e_i }$} (m-1-2);
\draw[->] (m-1-2) to node {$\rho$} (m-2-2);
\draw[->] (m-1-1) to node {$\sigma_i$} (m-2-1);
\draw[->](m-2-2) to node {$L_i \otimes F_{+\alpha_i  e_i } $}(m-2-1);
\end{tikzpicture}
\end{center}
where $\sigma_i$ is multiplication by the section $s_i$

\item \label{ex2} For all $i\neq j$ and $\alpha$ with $\alpha_i=r_i$ the following 
diagram commutes
\begin{center}
\begin{tikzpicture}
 \matrix(m)[matrix of math nodes, row sep=4em, column sep=7em]
 {F_\alpha & L_i \otimes F_{\alpha-r_ie_i}\\
 F_{\alpha+e_j} & L_i \otimes F_{\alpha+e_j-r_ie_i} \\};
 \draw[->] (m-1-1) to node[left] {$F_{\vec{e_j}}$} (m-2-1);
 \draw[->] (m-1-2) to node[right] {$F_{\vec{e_j}}$} (m-2-2);
 \draw[->] (m-1-1) to node {$\rho$} (m-1-2);
 \draw[->] (m-2-1) to node {$\rho$}(m-2-2);
\end{tikzpicture}
\end{center}

\item \label{ex3} For all $i$ and $j$ and $\alpha \in  \vr I^n $ with 
$\alpha_i=r_i$ and $\alpha_j=r_j$ the following diagram commutes 
\begin{center}
\begin{tikzpicture}
\matrix (m) [matrix of math nodes, row sep=4em, column sep=7em]
{F_\alpha & L_i \otimes  F_{\alpha-r_ie_i} \\
L_j \otimes F_{\alpha -r_je_j} & L_i \otimes L_j \otimes F_{\alpha-r_ie_i-r_je_j} \\};
 \draw[->] (m-1-1) to node[left] {$\rho$} (m-2-1);
 \draw[->] (m-1-2) to node[right] {$\rho$} (m-2-2);
 \draw[->] (m-1-1) to node {$\rho$} (m-1-2);
 \draw[->] (m-2-1) to node {$\rho$}(m-2-2);
\end{tikzpicture}
\end{center}

\end{enumerate}
\end{definition}

\begin{definition}
An extendable pair $(F, \rho)$ is called \textit{coherent} if for each $v \in \vr I^n$ the sheaf $F_v$ is a coherent sheaf on $X.$
\end{definition}

\begin{proposition}
Let $X$ be a scheme and $L$ a symmetric monoidal functor as in Definition \ref{parabolic}. Let $(E,\rho)$ be a parabolic sheaf on $(X,L)$ with denominators $\vr$. Then the restricted
functor $E|_{\vr I^n}$ produces an extendable pair on $(X,L)$.
\end{proposition}

\begin{proof}
Note that the restricted functor has all the required data for an extendable pair by restricting
the collection $\rho_{\alpha, \beta}$. We need to check that the axioms of an extendable pair are
satisfied. 

(\ref{ex1}) We have that the composition
\[
E_{\alpha+(r_i-\alpha_i)e_i} \stackrel{\rho}{\rightarrow} E_{\alpha-\alpha_ie_i}\otimes L_i 
\rightarrow E_\alpha\otimes L_i\stackrel{\rho^{-1}}{\rightarrow} E_{\alpha+r_i e_i}
\] 
is just the morphism $E_{+\alpha_i e_i}$ using axiom (ii) of parabolic sheaf. Precomposing
with the map 
\[
E_{+(r_i-\alpha_i)e_i}: E_\alpha \rightarrow E_{\alpha +(r_i-\alpha_i)e_i}
\]
gives the morphism $E_{+r_ie_i}$. The result now follows from axiom (i).

(\ref{ex2}) This follows directly from axiom (ii).

(\ref{ex3}) This follows directly from axiom (iii).
\end{proof}

\begin{proposition}
Let $X$ be a scheme and $L$ a symmetric monoidal functor as in Definition \ref{parabolic}. Given an extendable pair $(F,\rho)$ on $(X,L)$ we can extend it to a parabolic sheaf $(\hat{F},\rho)$ on $X,L$ and the extension is unique up to a canonical isomorphism. A coherent extendable pair extends to a coherent parabolic sheaf. 
\end{proposition}

\begin{proof}
For $v \in \ZZ^n$ we need to define its extension $\hat{F}_v$. We can write
$v_i = r_i u_i + q_i $ with $0\le q_i < r_i$ and $u_i \in \ZZ$. As before we denote $L_u = \otimes_{i=1}^n L^{\otimes u_i} $ and $q = (q_1,\ldots, q_n)$.
Set $\hat{F}_v = L_u \otimes F_q$.

We need to construct maps $\hat{F}_{+e_i}:\hat{F}_v\rightarrow \hat{F}_{v+e_i}$.
If $q_i< r_i -1 $ then the map is obtained by tensoring the map
$F_{q_i} \rightarrow F_{q_i + e_i }$ with $L_u$.
 If $q_i =r_i-1$
then the map is defined by
\begin{center}
\begin{tikzpicture}
\matrix (m) [matrix of math nodes, row sep=3em, column sep = 3em]
{\hat{F}_v = L_u \otimes F_q &   &\hat{F}_{v+ e_i}=L_u \otimes L_i \otimes F_{q'}\\
& L_u \otimes F_{q+e_i} & \\};
\draw [->] (m-1-1) to node {$\hat{F}_{e_i}$} (m-1-3);
\draw [->] (m-1-1) to node[below, yshift=-0.11cm, xshift=-0.15cm] {$1 \otimes F_{\vec{e_i}}$} (m-2-2);
\draw[->] (m-2-2) to node[below, yshift=-0.11cm, xshift=0.05cm]{$1 \otimes \rho$} (m-1-3);
\end{tikzpicture}
\end{center}
where $q'_j = q_j$ for all $j \neq i$ and $q'_i = 0.$

In order to show that the construction above indeed produces a functor we need to show that
all diagrams of Lemma \ref{l:data} commute. If both $q_i<r_i-1$
and $q_j<r_j-1$ then this is straightforward.
Suppose that $q_i=r_i-1$
and $q_j<r_j-1$ then this follows from (EX\ref{ex2}). This leaves
the case $q_i=r_i-1$
and $q_j=r_j-1$. We have a diagram
\begin{center}
\begin{tikzpicture}
\matrix (m) [matrix of math nodes, row sep=3em, column sep = 1.5em]
{
L_u\otimes F_q & L_u\otimes  F_{q+e_i} & L_u\otimes L_i \otimes F_{q - q_i e_i} \\
L_u\otimes F_{q+e_j} & L_u\otimes F_{q+e_i+e_j} & L_u\otimes L_i \otimes F_{q - q_i e_i + e_j}\\
L_u\otimes L_j \otimes F_{q-q_je_j} & L_u\otimes L_j \otimes F_{q-q_je_j +e_i} &
L_u\otimes L_i \otimes L_j \otimes F_{q-q_ie_i-q_je_j}\\
};


\draw [->] (m-1-1) to (m-1-2);
\draw [->] (m-1-2) to (m-1-3);

\draw [->] (m-2-1) to (m-2-2);
\draw [->] (m-2-2) to (m-2-3);

\draw [->] (m-3-1) to (m-3-2);
\draw [->] (m-3-2) to (m-3-3);


\draw [->] (m-1-1) to (m-2-1);
\draw [->] (m-2-1) to (m-3-1);

\draw [->] (m-1-2) to (m-2-2);
\draw [->] (m-2-2) to (m-3-2);

\draw [->] (m-1-3) to (m-2-3);
\draw [->] (m-2-3) to (m-3-3);

\end{tikzpicture}
\end{center} 
The top left square commutes using the fact that $F $ is a functor.
The top right and bottom left squares commute using axiom (EX\ref{ex2}).
The bottom right square commutes using axiom (EX\ref{ex3}). So indeed $\hat{F}_{\bullet}$ is a functor.

Note that we have canonical isomorphisms $L_u\otimes L_v\cong L_{u+v}$ for 
$u,v\in \vr\ZZ$. These isomorphisms induce our pseudo-period isomorphisms.

 Finally we need to check the conditions (i) to (iv) of a parabolic sheaf.  
 
 \noindent
 Condition (i):  For $\vr \alpha, \vr \alpha' \in \vr \mathbb{N}^n $ the following diagram commutes

  \begin{center}
\begin{tikzpicture}
\matrix (m) [matrix of math nodes, row sep=3em, column sep = 2.5em]
{
\hat{F}_v & \hat{F}_{v+\vr \alpha} &  \hat{F}_{v + \vr \alpha +\vr \alpha'}  \\
  & L_{\alpha}\otimes \hat{F}_{v} & L_{\alpha}\otimes L_{\alpha'} \otimes \hat{F}_{v} \\
};

\draw [->] (m-1-1) to node {$\hat{F}_{+\vr \alpha} $} (m-1-2);
\draw [->] (m-1-2) to node {$\hat{F}_{+\vr \alpha'} $} (m-1-3);
\draw [->] (m-2-2) to (m-2-3);

\draw [->] (m-1-2) to  (m-2-2);
\draw [->] (m-1-3) to  (m-2-3);

\end{tikzpicture}
\end{center} 
 
 This follows by the definition of the functor $\hat{F}_{\bullet}$ and  the symmetric monoidal structure of $L$.
 
 This allows us to make the following reduction: in order to check axiom (i) it suffices to check that the following diagram commutes

  \begin{center}
\begin{tikzpicture}
\matrix (m) [matrix of math nodes, row sep=3em, column sep = 2.5em]
{
  \hat{F}_v &  \hat{F}_{v + r_ie_i }  \\
 & L_{i}\otimes \hat{F}_{v} \\
};


\draw [->] (m-1-1) to (m-1-2);


\draw [->] (m-1-2) to node {$\rho$} (m-2-2);
\draw [->] (m-1-1) to node[below]{$\sigma_i$} (m-2-2);

\end{tikzpicture}
\end{center} 
 
 And this follows directly  from (EX\ref{ex1}).
 
 \noindent
 Condition (ii): Once again we reduce to showing that 

  \begin{center}
\begin{tikzpicture}
\matrix (m) [matrix of math nodes, row sep=3em, column sep = 2.5em]
{
  \hat{F}_{v+r_ie_i } &  L_{i} \otimes  \hat{F}_{v}  \\
   \hat{F}_{v+ b + r_ie_i} & L_{i}  \otimes \hat{F}_{v +b} \\
};


\draw [->] (m-1-1) to node{} (m-1-2);
\draw [->] (m-2-1) to (m-2-2);

\draw [->] (m-1-2) to node{$L_i \otimes \hat{F}_{+b}$}  (m-2-2);
\draw [->] (m-1-1) to node{$\hat{F}_{+b}$} (m-2-1);

\end{tikzpicture}
\end{center} 
commutes. If we write $v = \vr u + q $  then this diagram will become

 \begin{center}
\begin{tikzpicture}
\matrix (m) [matrix of math nodes, row sep=3em, column sep = 2.5em]
{
 L_{u+e_i} \otimes F_{q} &  L_{i} \otimes (L_u \otimes  F_q)  \\
  L_{u+e_i} \otimes  \hat{F}_{q+ b} & L_{i}  \otimes (L_u \otimes \hat{F}_{q +b}) \\
};


\draw [->] (m-1-1) to node{} (m-1-2);
\draw [->] (m-2-1) to (m-2-2);

\draw [->] (m-1-2) to node{$L_i \otimes L_u \otimes \hat{F}_{+b}$}  (m-2-2);
\draw [->] (m-1-1) to node{$ L_{u+e_i} \otimes \hat{F}_{+b}$} (m-2-1);

\end{tikzpicture}
\end{center} 

  We can use the symmetric monoidal structure of $L$ to show that this diagram indeed commutes.

\noindent
 Condition (iii): We reduce to showing the  commutativity of the following diagram

  \begin{center}
\begin{tikzpicture}
\matrix (m) [matrix of math nodes, row sep=3em, column sep = 2.5em]
{
  \hat{F}_{v+r_ie_i + r_je_j} &  L_{i} \otimes  \hat{F}_{v+ r_je_j}  \\
  L_{j} \otimes  \hat{F}_{v+ r_ie_i} & L_{i}\otimes L_j \otimes \hat{F}_{v} \\
};


\draw [->] (m-1-1) to node{} (m-1-2);
\draw [->] (m-2-1) to (m-2-2);

\draw [->] (m-1-2) to node {} (m-2-2);
\draw [->] (m-1-1) to node[below]{} (m-2-1);

\end{tikzpicture}
\end{center} 
  which  follows from the monoidal structure of $L$.

\noindent
 Condition (iv) is by definition. 

  Finally, let  $E_{\bullet}$ be another extension  of $F_{\bullet}. $ Again we can again write $v_i =r_i u_i + q_i $ with $0\le q_i < r_i$ and $u_i \in \ZZ.$ By pseudo-periodicity, $E_v \simeq L(u) \otimes E_q, $ and  $F_q = E_q $ because  $E_{\bullet}$ is an extension.  So, $E_v \cong \hat{F}_v$ for any $v \in \ZZ^n.$ 

 It is clear from the construction that finitely generated condition is preserved under extension. 
\end{proof}

\begin{corollary} 
\label{extension}
Let $X$ be a scheme and $L$ a symmetric monoidal functor as in Definition \ref{parabolic}. The category of parabolic sheaves  (coherent parabolic sheaves) on $(X,L)$ is equivalent to the category of extendable pairs (resp. coherent extendable pairs ) on $(X,L)$. 
\end{corollary}

\begin{proof}
There is a pair of functors between these categories. The truncation functor sends a parabolic sheaf $(E, \rho)$ to an extendable pair by forgetting all $E_v$ when $v \notin \vr I^n.$ And the extension functor from extendable pairs to parabolic sheaves was defined in the previous Proposition on objects by the rule $F_{\bullet} \mapsto \hat{F}_{\bullet}. $ It is easy to see that these functors are mutually inverse and preserve finitely-generation condition. 
\end{proof}

\begin{remark}
	\label{CohExt}
 Let $X$ be a scheme and $L$ a symmetric monoidal functor as in Definition \ref{parabolic}. We will denote the  category of coherent extendable pairs on $(X,L)$  by $\mathcal{EP}(X,L,\vr).$ When $X$ is locally noetherian this category is abelian.
\end{remark}

\subsection{The localization sequence}
\label{s:loc-seq}

 In this subsection we will localize the category of finitely-generated extendable pairs so that it will be glued from  simpler parts. 
 
 For this section $X$ is a locally noetherian scheme and $L$ a symmetric monoidal functor as in Definition \ref{parabolic}. T
 
First let us consider the functor $\pi_*^{L,\vr}: \mathcal{EP}(X,L,\vr) \rightarrow \mathfrak{Coh}X, $ given by $F_{\bullet} \mapsto F_{0} $ on objects. It is  an exact functor because exact sequences in diagram categories are defined point-wise. 
  \begin{lemma}
  \label{epsil}
   The functor $\pi_*^{L,\vr}$ has  a left adjoint denoted $\pi^*_{L,\vr}$ and there is a natural isomorphism 
$\pi_*^{L,\vr}   \circ  \pi^*_{L,\vr} \simeq 1$.
  \end{lemma}
\begin{proof}
In what follows, we will omit,  the superscripts (resp. subscripts) $L$ and  $\vr$ in the notation for the appropriate functors. For any $0 \leq i \leq n $ consider functions $\epsilon_i: \vr I \rightarrow \{0,1\}, $ defined by $\epsilon_i (u) = 1 $ if $u_i = r_i$ and zero otherwise. We define the  functor  $\pi^*$ on a sheaf $F \in \mathfrak{Coh}X$ by the rule: 
$$(\pi^*(F))_{u} = (\otimes_{i = 1}^n L_i^{\epsilon_i(u)})\otimes F  . $$
 This forms a functor via the maps
 
 $$(\pi^*(F))_{u} \rightarrow (\pi^*(F))_{u +e_i} = \begin{cases}
\text{identity} & \text{ if } u_i \in [0,r_i-2] \\
\sigma_i  & \text{ if } u_i = r_{i}-1, 
\end{cases}$$
where $\sigma_i$ is the multiplication by the section $s_i$. 

Define $\rho$ to be identity map. It is easy to see that all axioms of  extendable pair are satisfied. 

   Now let's take a coherent sheaf $F$ and an extendable pair $E_{\bullet}$ and consider a map 
   
   $$ \text{Hom}_{\mathfrak{Coh}X}(F, \pi_*E)  \rightarrow \text{Hom}_{\mathcal{EP}} (\pi^*F, E)   $$
 given by sending $\phi \in \text{Hom}_{\mathfrak{Coh}X}(F, \pi_*E)$ to precomposition of the structure maps of the extendable pair $E$ with $\phi.$  It's obviously an injection. Surjectivity will follow from commutativity of the squares in $\text{Hom}_{\mathcal{EP}} (\pi^*F, E) $ and because all structure maps in $ \pi^*F$ are identity.

\end{proof}

\begin{proposition} Let $X$ be a locally noetherian scheme.
The functor $ \pi_*^{L,\vr} :  \mathcal{EP}(X,L,\vr) \rightarrow \mathfrak{Coh}X $ satisfies the hypothesis of Theorem \ref{t:swanQuot}. 
\end {proposition}
\begin{proof}
 The only thing which is not completely obvious is the second condition. Consider two extendable pairs $E_{\bullet}$ and $F_{\bullet}.$ Suppose that we have a
 morphism $\pi_* (E_{\bullet}) \rightarrow \pi_*(F_{\bullet})$. By adjointness we obtain a diagram
 \begin{center}
 \begin{tikzpicture}
 \node (T) at (0,0) {$\pi^* \pi_* (E_{\bullet})$ } ;
 \node (B) at (0,-1.5) {$E_{\bullet}$} ;
 \node (BR) at (2,-1.5) {$F_{\bullet}$} ;
 \draw[->] (T) -- (B);
 \draw[->] (T) -- (BR);
 \end{tikzpicture}
 \end{center}
 Applying $\pi$ to this picture shows that the second condition holds.
\end{proof}
 
 Using  Theorem \ref{t:swanQuot} we obtain the following
 
 \begin{corollary}
 \label{1localization} Let $X$ be a locally noetherian scheme.
 There is an equivalence of abelian categories: 
    $$  \mathcal{EP}(X,L,\vec{r}) \Big/ {\bker( \pi_*^{L,\vec{r}})} \rightarrow  \mathfrak{Coh}X,$$
   \end{corollary}

     In the rest of this subsection we would like to give a description of the category ${\bker(\pi_*^{L,\vec{r}})}.$  Let us study the objects first. Let $F_{\bullet}$ be an extendable pair. Then $\pi_* (F_{\bullet}) = F_0,$ and if $F_{\bullet} \in {\bker(\pi_*^{L,\vec{r}})} $ then $F_0 \cong 0.$ The pseudo-period isomorphism imply in turn that  $F_u \cong 0$ if  all $u_i \in \{ 0, r_i \}.$
       
    Let us consider the sheaves $F_u$ such that for any $j \neq i $ $u_j \in \{ 0, r_j \} $ (we can imagine them as sheaves on the edges of the cubical diagram $F_{\bullet} \in \text{Func}(\vr I^n, \bA) $).  
    Using the axiom (EX \ref{ex1})  we get that the multiplication by section map
     $s_i: F_u \rightarrow L_i \otimes F_u$  must factor through $F_{u+ (r_i-u_i)e_i}$ which is a zero sheaf if $F_{\bullet} \in {\bker(\pi_*^{L,\vec{r}})}. $ This implies  the following:

   \begin{lemma}
   \label{support}
 If $F_{\bullet} \in {\bker(\pi_*^{L,\vec{r}})} $ and $u \in \vr I^n $ is such that $\forall j \neq i $  $ u_j \in \{ 0, r_j \}$ then   $\supp(F_u) $ is contained in the divisor of zeroes of the section
 $s_i\in H^0(L_i)$.
 
 If $s_i = 0$ for some $i$, we will say that $\div(s_i) = X$.
   \end{lemma}

    We will apply the localization method (Theorem \ref{t:swanQuot}), to this partial description of the
    kernel.   
    
    Let's  fix some notation.  Denote by 
   $$S_n(k) = \{T \subset \{1, \dots, n\}\mid |T|=k \}.$$
   We will often abuse notation and write $S(k)$ for $S_n(k)$ when it is clear from the
   context what $n$ is.
We will view each interval $[0,r_i]$ as a pointed set, pointed at $0$. It
follows that we have order preserving inclusions
\[
\iota_T : \prod_{i \in T} [0,r_i] \rightarrow \prod_{i=1}^n [0,r_i] := \vr I^n.
\]   
 
 Ignoring the pointed structure produces order preserving ($\le$) projection maps
 \[
 \pi_T : \vr I^n \rightarrow \prod_{i\in T} [0,r_i].
 \]  

\begin{definition}
 As we agreed above, $L: \vr\ZZ^n \rightarrow \mathfrak{Div}X$ is the symmetric monoidal functor as in Definition \ref{parabolic}.
  
    If $1 \le k \le n$   and   $T \in S(k)$ then we will define a symmetric monoidal functor $L_T:   \vr \ZZ^k \rightarrow  \mathfrak{Div}X $ as a composition:
  
  \begin{center}
\begin{tikzpicture}
\matrix (m) [matrix of math nodes, row sep=3em, column sep = 2.5em]
{
 \vr\ZZ^k &   \vr\ZZ^n   &  \mathfrak{Div}X \\
};

\draw [->] (m-1-1) to node {$ \iota_T $} (m-1-2);
\draw [->] (m-1-2) to node {$L $} (m-1-3);
\end{tikzpicture}
\end{center} 
 We will say that $L_T$ is obtained from $L$ by the \textit{restriction along} $\iota_T.$
\end{definition}

 Now for $T \in S(k)$ let's consider the functor 
 
 $$\iota_T^*  :  \mathcal{EP}(X,L,\vec{r})  \longrightarrow  \mathcal{EP}(X,L_T, \pi_T (\vr) )  
      $$
     which is the restriction of  an extendable pair $F_\bullet$ along the inclusion
     $\iota_T$.
The pseudo-period isomorphism is just obtained by restriction. 

\begin{definition}\label{d:face}
   For any $1 \le k \le n$ we define functors
 $$\face^k := \prod_{T \in S(k) } \iota_T^*  :   \mathcal{EP}(X,L,\vec{r})  \longrightarrow \prod_{T \in S(k) }  \mathcal{EP}(X,L_T, \pi_T (\vr) ) 
 $$
\end{definition}

\begin{definition}\label{d:ker}
\label{ker-k}
 For any $1 \le k \le n$  we denote by  $ \bker^k = \bker(\face^k) $. 
 
 Also denote $\bker^0 = \bker(\pi_*). $

 \end{definition}  
 
  \begin{lemma}
  \label{l:sup-k}
  For any $1 \le k \le n$, any $F_\bullet \in \bker^{k-1}$ and any $T\in S(k)$ we can consider   $(\iota_{T}^*(F_{\bullet}))_\bullet$ as an element of 
  $$ \func( \prod_{i \in T} [1,r_i - 1], \phantom{a} \mathfrak{Coh} (\bigcap_{i \in T}\div(s_i))) .$$  
  In other words, the images of these functors are supported on the indicated subschemes.
  As in Lemma \ref{support} we will say that if $s_i = 0$, then $\div(s_i) = X$.
 \end{lemma}
  
 \begin{proof}
If $k = 1$ then the result is proved in  Lemma \ref{support} and the observation before it.
 
 Let's take any $2 \le k \le n$  and an extendable pair $F_\bullet \in \bker^{k-1}.$

 If we consider an extendable pair $(\iota^*_{T}(F_{\bullet}))_\bullet \in  \mathcal{EP}(X,L_T, \pi_T(\vec{r})) $ then 
 for any $v \in  \prod_{i \in T} [0, r_i]$ we will have isomorphisms of sheaves:  $(\iota^*_{T}(F_{\bullet}))_v \cong 0,$ whenever $v_i = 0$ for some $i \in T$. Because of the pseudo-periodicity isomorphism we also have that $(\iota_{T}(F_{\bullet}))_v \cong 0,$ whenever $v_i = r_i$ for some $i \in T$. 
  
   The last step is an application of the axiom EX1 to the extendable pair $(\iota^*_{T}(F_{\bullet}))_\bullet.$ Because 
    $(\iota^*_{T}(F_{\bullet}))_v \cong 0$ if $v_i = r_i$ for some $i \in T$ that implies that for  any $w \in \prod_{i \in T} [1, r_i - 1] $   the multiplication of the sheaf $(\iota^*_{T}(F_{\bullet}))_w $ by the sections $s_i \in H^0(X, L_i)$ for all $i \in T$ must factor through zero. So the support of the sheaf $(\iota^*_{T}(F_{\bullet}))_w $  is contained in 
    $\cap_{i \in T}\div(s_i).$
 
\end{proof}

  \begin{lemma}
  If we restrict the domain of the functor $\face^k$ to the full subcategory $\bker^{k-1}$ for any $1 \leqslant k \leqslant n,$ then we will obtain functors:
  $$ \face^k \big |_{\bker^{k-1} } : \bker^{k-1}  \longrightarrow \prod_{T \in S(k) } \func(\prod_{i \in T} [1,r_i - 1], \phantom{a} \mathfrak{Coh}( \bigcap_{i \in T} \div(s_i))).
   $$ 
There is an equivalence of categories between $\bker^k$ and $\bker(\face^k\big |_{\bker^{k-1} })$.
  \end{lemma}
 
 \begin{proof}
 The first part follows directly from the Lemma before. The proof of the second part is straightforward and follows from the fact that $\bker^k$ is a full subcategory of $\bker^{k-1}.$ 
\end{proof}  
  
  \begin{remark}
  In order to apply the localization procedure to the category $\bker^{k-1}$ we need to show that the functor $ \face^k \big |_{\bker^{k-1} }$  has a left adjoint. The existence of a left adjoint follows from the special adjoint functor theorem. But for the purpose of splitting of the corresponding short exact sequence of $K$-groups (see the section \ref{ss:computation} for details) we need the unit of the adjunction to be the natural isomorphism. This doesn't follow from the abstract nonsense, so we need an explicit construction of a left adjoint functor. It is given in the proof of  the following theorem. 
  
  \end{remark}

  \begin{theorem}
\label{localization} Let $X$ be a locally noetherian scheme and consider a symmetric monoidal functor
$L:\vr\ZZ^n \rightarrow  \Div X$.
 \begin{enumerate} [(i)]
\item
For any $1 \le k \le n$ there is an exact functor 
$$\face^k \big |_{\bker^{k-1} } : \bker^{k-1}  \longrightarrow \prod_{T \in S(k) } \func(\prod_{i \in T} [1,r_i - 1], \phantom{a} \mathfrak{Coh}( \bigcap_{i \in T} \div(s_i))), $$
where $ \bker^k$ is a kernel of the functor $\face^k$   and $\bker^0 := \bker(\pi_*^{L, \vr}) $
\item
The functors  $\face^k\big |_{\bker^{k-1} }$ have left adjoints $D^k $ such that 
$$\face^k \big |_{\bker^{k-1} } \circ D^k \simeq 1$$. 
\item
 $\face^k \big |_{\bker^{k-1} }$  satisfies the condition of Theorem \ref{t:swanQuot}
\item
The functor 
$$\face^n\big |_{\bker^{n-1} }:  \bker^{n-1}  \longrightarrow \func(\prod_{i =1}^n [1, r_i-1], \phantom{a} \mathfrak{Coh}( \bigcap_{i =1}^n \div(s_i)))$$
 is an equivalence of categories.
\end{enumerate}
\end{theorem}    
  
  \begin{proof}
  
  (i) These functors are obtained by restricting domains. As kernels and cokernels are computed pointwise
  this is exact.
  
  (ii)  
  Given a functor  $G^{T}_{\bullet}\in \text{Func}(\prod_{i \in T} [1,r_i - 1], \phantom{a} \mathfrak{Coh}( \bigcap_{i \in T} \div(s_i))) $ for each
  $T \in S(k) $, we will denote the corresponding object:  

 $$(G^{T}_\bullet)_{T \in S(k) } \in \prod_{T \in S(k) } \func(\prod_{i \in T} [1,r_i - 1], \phantom{a} \coh( \bigcap_{i \in T} \div(s_i)) )  .$$ 
  
   Further we will view $G^{T}_\bullet$ as a functor $\prod_{i \in T}[0,r_i] \rightarrow  \mathfrak{Coh}( \bigcap_{i \in T} \div(s_i))$
  by taking $G^{T}_u = 0$ if for some $i \in T$   we have $u_{i} \in \{0, r_i \} , $ where $0$ is some fixed zero object in $\coh(X)$. Also for $i \in \{1, \dots, k \}$  if $u_{i} \in \{0, r_{i} - 1 \}$ we define the morphisms $G^{T}_{+e_i}: G^{T}_u \rightarrow G^{T}_{u+e_i}$  as the initial and terminal map correspondingly.

  Let us remind the reader of the definition of $\epsilon$ from  Lemma \ref{epsil}. For any $0 \leq i \leq n $ we have functions $\epsilon_i: \vr I \rightarrow \{0,1\}, $ such that for any $u \in \vr I^n$: $\epsilon_i (u) = 1 $ if $u_i = r_i$ and zero otherwise.

  We define the functor $D^k$ on objects as follows:
         
                $$(D^k((G^{T}_\bullet)_{T \in S(k) }))_u =  (\otimes_{i=1}^n L_i^{\epsilon_i(u)}) \otimes (\oplus_{T \in S(k)} G^{T}_{\pi_T(u)} ) $$ 
                
          Let's denote $(D^k((G^{T}_\bullet)_{T \in S(k) }))_\bullet $ by $D^k_\bullet$ for the simplicity of notations. First of all we want to view it as a functor $\vr I^n \rightarrow \mathfrak{Coh}(X).$ For that we have to define the morphisms:

    $$D^k_{+e_i}:   D^k_u  \rightarrow D^k_{u+e_i } . $$
  If $0 \leqslant  u_i < r_i - 1, $ then this map is induced by $\oplus_{\substack{T \in S(k) \\ \text{s.t.} \phantom{a}i \in T }} G^{T}_{+1}.$  If $u_i = r_i-1,$
  then it is induced by the terminal maps $\oplus_{\substack{T \in S(k) \\ \text{s.t.} \phantom{a}i \in T }} G^{T}_{+1}.$  and also by multiplication by the section $s_i.$ 
  
           The pseudo-period isomorphisms $\rho$ are defined by the symmetric monoidal structure of the functor $L.$
The proof of the axioms EX2 and EX3 is automatic. And the proof of EX1 will follow from the commutativity of the diagram:

 \begin{center}
\begin{tikzpicture}
\matrix (m) [matrix of math nodes, row sep=4em, column sep=8em]
{D_u & D_{u+ (r_i- u_i)e_i}  \\
   L_i \otimes D_{u}& L_i \otimes D_{u - u_i  e_i}. \\};
\draw[->] (m-1-1) to node {$D_{+(r_i - u_i) e_i}$} (m-1-2);
\draw[->] (m-1-2) to node {$\rho$} (m-2-2);
\draw[->] (m-1-1) to node {$\sigma_i$} (m-2-1);
\draw[->](m-2-2) to node {$L_i \otimes D_{+u_i  e_i} $}(m-2-1);
\end{tikzpicture}
\end{center}

    This diagram will commute because of the definition of $D_{+(r_i - \alpha_i) \vec{e_i}}$ and because $\text{supp}(G^T_{u}) \subseteq \cap_{i \in T}\div(s_i)$ for any $u \in \prod_{i \in T} [0, r_i]$.

 So we have shown that $D^k_\bullet$ is an extendable pair.  If $k = 1$ then it's clear that $D^1_\bullet$ is in $\bker^0,$ because $D^1_{0} \cong 0.$ 
 
 If $2 \leqslant k \leqslant n, $ we want to see that $D^k_\bullet$ is in $\bker^{k-1}.$ For that we have to see that for any $W \in S(k-1)$ and any $v \in \prod_{i \in W}[0, r_i]$ the sheaf $(\iota^*_{W}(D^k_\bullet))_v $ is isomorphic to zero. But this is true because for any $T \in S(k)$ we have that $G^{T}_u = 0$ if $u_i \in \{0, r_i\} $ for some $i \in T.$
            
  Clearly, $\text{Face}^k \big |_{\bker^{k-1} } \circ D^k  = 1.$

  Next we would like to show that $D^k$ is indeed a left adjoint.  Suppose that we have a morphism 
  $$(G^{T}_\bullet)_{T \in S(k) } \rightarrow \face^k(F_{\bullet}). $$

Such a morphism consists of an ${n \choose k}$-tuple of morphisms
$$ \phi_{T} : G^{T}_\bullet \rightarrow \iota^*_{T}(F_\bullet). $$
We wish to describe the adjoint map
$$
\tilde\phi : D^k_\bullet \rightarrow F_\bullet.  
$$
Using the universal property of coproduct, this morphism is determined by 
maps
\[
\tilde\phi (u)_{T}:  \otimes_{i=1}^n L_i^{\epsilon(u)}\otimes G^{T}_{\pi_T(u)}\rightarrow F_u.
\]

If $u$ is such that $\epsilon_i(u) = 0$ for all $1\leq i \leq n,$ then these maps are just the compositions of  $\phi_{T}$ with the morphisms $F_{+\alpha}.$ If there are $l$'s, such that $u_l = r_l,$ then $\tilde\phi (u)_{T}$ is induced by the composition of $\phi_{T}$ with $\rho^{-1}_F$  and with $F_{+\alpha}.$ 

 We want to check that the map $ \tilde\phi $ is indeed a natural transformation of functors. It's enough to check that the diagram commutes:
  
   \begin{center}
\begin{tikzpicture}
\matrix (m) [matrix of math nodes, row sep=4em, column sep=8em]
{D_u & F_u  \\
   D_{u +e_i}& F_{u + \vec{e_i}} \\};
\draw[->] (m-1-1) to node {$ \tilde\phi (u)   $} (m-1-2);
\draw[->] (m-1-2) to node {$F_{+ e_i}$} (m-2-2);
\draw[->] (m-1-1) to node {$D_{+ e_i}$} (m-2-1);
\draw[->](m-2-1) to node {$\tilde\phi (u + e_i) $}(m-2-2);
\end{tikzpicture}
\end{center}
   
 If $\epsilon_k (u) = 0$ for all $1 \leq k \leq n$  and also $u_i < r_i - 1,$ then it commutes directly from the construction of the maps $  \tilde\phi (u) .  $ Otherwise the commutativity will follow from EX1, EX2 and EX3 for $F_\bullet$.

  Finally we obtained the map: 
  
  $$ \text{Hom}((G^{T}_{\bullet})_{T \in S(k) }, \text{Face}^k(F_{\bullet}) ) \rightarrow \text{Hom}(D^k((G^{T}_{\bullet})_{T \in S(k) }), F_{\bullet} )  .   $$
It's easy to see that this map is bijective, because  the right $\Hom$ is uniquely defined by the restriction to  $k$-faces.

  (iii) Follows from (ii).
  
  (iv) Because for $S(n)$ there is only one element, the set $\{1, \dots, n \}$ itself,  we have that $\iota_{\{1, \dots, n \}} = id$ and $\pi_{\{1, \dots, n \}} = id.$ So $\text{Face}\big |_{\bker^{n-1} }^n$ and $D^n$ are identity functors. 
  
  \end{proof}

  \subsection{$G$-theory and $K$-theory of a root stack}\label{ss:computation}

        In this subsection we will finally describe the $G$-theory of a root stack $X_{L,\vr}$.
        
        \begin{lemma}
        	\label{l:1000}
        	If $X$ is a locally noetherian scheme and $L$ a symmetric monoidal functor as in Definition \ref{parabolic}, there is an equivalence of categories:
        	$$ \mathfrak{Coh} X_{L,\vr} \simeq  \mathcal{EP}(X,L,\vr)$$
        \end{lemma} 
        \begin{proof}
        	This follows by combining  Corollaries \ref{equivalence} and  \ref{extension}.
        \end{proof}

       So we  have:
        $$ G(X_{L,\vr})  \cong K( \mathcal{EP}(X,L,\vr)), $$  
       and we reduced the problem to describing the $K$-theory of the (abelian) category of coherent extendable pairs $\mathcal{EP}(X,L,\vec{r})$.

 We are going to use several splittings of the category of coherent extendable pairs to simplify the latter $K$-theory. The first step is this
 
  \begin{lemma}
  \label{l:K1} 
  If $X$ is a locally noetherian scheme, then in notation of section \ref{s:loc-seq} one has:
 $$ K_i( \mathcal{EP}(X,L,\vr)) \cong G_i(X) \oplus K_i(\bker( \pi_*^{L,\vec{r}})) \  \textnormal{for any} \ i \in \mathbb{Z}_+$$ 
  \end{lemma}
 
\begin{proof}
Using Corollary \ref{1localization} and the localization property of $K$-theory (see for example \cite{Q})   we have the long exact sequence of groups: 
 $$ \dots\rightarrow K_i(\bker( \pi_*^{L,\vec{r}}) ) \rightarrow K_i(\mathcal{EP}(X,L,\vr)) \rightarrow G_i(X) \rightarrow \dots  $$
 
But this sequence splits because of the property $\pi_*^{L,\vr}   \circ  \pi^*_{L,\vr} \simeq 1$ proved in Lemma \ref{epsil}.
 \end{proof}
 
 Lets note :
  
\begin{lemma}  
If $\bA$ is an abelian category then
$$K_i(\func(\vr I^n, \bA)) \cong K_i(\bA)^{\oplus \prod_{j=1}^n r_j} . $$  
 \end{lemma}
  
\begin{proof}
  The proof follows from the iterated application of  Theorem \ref{standard} and localization property of the $K$-theory. 
\end{proof}

 Now we want to proceed with $K_{\bullet}(\bker( \pi_*^{L,\vec{r}})) $ exploiting same ideas as in the previous lemmas.
 
 \begin{lemma}
 	\label{l:K2}
 	Let $X$ be a locally noetherian scheme,  $L$ a symmetric monoidal functor as in Definition \ref{parabolic}, and $s_k \in H^0(L_k)$ for $k = 0,\ldots,n$. Then for any $i \in \mathbb{Z}_+$:
 	$$ K_i(\bker( \pi_*^{L,\vr})) \cong \oplus_{k=1}^n \oplus_{T \in S(k) } G_i(\bigcap_{l \in T} \div(s_l))^{\oplus \prod_{l \in T}(r_l - 1)} , $$  
 	where $S(k) = \{T \subset \{1, \dots, n\}\mid |T|=k \}$.
  \end{lemma}
 
 \begin{proof}
 Follows from application of the localization property of $K$-theory, Theorem \ref{localization} and the previous technical lemma.
 \end{proof}

  Combining Lemmas \ref{l:1000}, \ref{l:K1} and  \ref{l:K2} one yields the main result of the section:
  
  \begin{theorem}
  	\label{th:G-root}
  	Let $X$ be a locally noetherian scheme. Let $(L_i, s_i)$ be objects of $\mathfrak{Div}X $ for $i = 1, \dots, n$, and $\vr \in \NN^n$. Then $G$-theory of a root stack $X_{L, \vr} $ is given by the formula:
  	$$G_i(X_{L,\vr})  \cong G_i(X) \oplus \left(\oplus_{k=1}^n \oplus_{T \in S(k) } G_i(\bigcap_{l \in T} \div(s_l))^{\oplus \prod_{l \in T}(r_l - 1)}\right) $$
  	for any $i \in \mathbb{Z}_+$, and where $S(k) = \{T \subset \{1, \dots, n\}\mid |T|=k \}$.
  \end{theorem}

  To finish the section we want to give sufficient conditions when a root stack is smooth. 
  
  \begin{proposition}
  	\label{p:regular}
  		Let $X$ be a smooth scheme over a field $k$. Let $D = \sum_{i=1}^n D_i$ be a  normal crossing divisor. Assume that $\vr$ is an $n$-tuple of natural numbers, such that each $r_i$ is coprime to the characteristic of $k$. Then a root stack $X_{D,\vr}$ is smooth.
   \end{proposition}
   
   \begin{proof}
   		By definition a stack is smooth if its presentation is a smooth scheme. The question is local, so we can assume that $X = \spec(R)$ and a divisor $D$ is a strict normal crossing divisor. If we localize further, we can assume that $R$ is a regular local ring, $D_i = (f_i)$ and $\{f_i\}$ form a part of a regular sequence of parameters. 
   		
   		By Example \cite[2.4.1]{C}, the presentation of a root stack $X_{D, \vr}$ is an affine scheme $A =  R[t_1, \dots, t_n]/(t_1^{r_1} - f_1, \dots, t_n^{r_n} - f_n)$. By \cite[Lemma 1.8.6]{GrM} this scheme is smooth.    	
   \end{proof}
  
  \begin{corollary} \label{c:kequalsg}
  	Under hypothesis of Proposition \ref{p:regular}, $G(X_{D,\vr}) = K(X_{D,\vr})$, where the latter means the Waldhausen $K$-theory of perfect complexes on the stack as defined in \cite{Joshua}.
  \end{corollary}
  
  \begin{proof}
  	Indeed, if a stack is regular, its Waldhausen $K$-theory is the same as $G$-theory. See \cite{Joshua}.
  \end{proof}

  \section{Quotient stacks as root stacks} \label{s:quotient}
  
 \subsection{Generation of inertia groups}
 
        Let $X$ be a scheme with an action of a finite group $G$. We will always assume that this action is \emph{admissible}. Let us recall, following \cite[V.1, Definition 1.7]{SGA1}, that an action is called admissible, if there exists an affine morphism $\phi: X \rightarrow Y$, such that $\sO_Y \cong \phi_*(\sO_X)^G $. This implies that the quotient $X/G$ exists and is isomorphic to $Y$.

 If $x\in X$ is a point (not necessarily closed)
 the subgroup of $G$ stabilising $x$ is called the \emph{decomposition group} and we denote it 
 by $D(x,G)$. The subgroup of the decomposition group acting trivially on the residue field of $x$ is
 called the \emph{inertia group} of $x$ and we denote it by $I(x,G)$. 
 
Note that there is an induced action of $D(x,G)$ on the closure of the point $x$ and $I(x,G)$ acts
trivially on this closure. Hence if $x\in\bar{y}$ then there is an inclusion 
$I(y,G)\hookrightarrow I(x,G)$. We will say that \emph{the inertia groups are generated in codimension
one} if for each point $x\in X$ we have that
\[
I(x,G) = \prod_{x\in \bar{y}} I(y,G)
\]
where the product is over all points of codimension one containing $x$ and the identification is
via the inclusions above. 
For a group acting on a smooth curve all inertia groups will be generated in codimension
one.
We will see that under certain assumptions that  this will be also true in higher dimensions (see Theorem \ref{t:inertia}).

\subsection{Main theorem}
 
 In this subsection we will provide sufficient conditions for a quotient stack to be a root stack. To illustrate the procedure we will start with an example.

 \begin{example} \label{e:dvr}
 Let $\sO$ be a discrete valuation ring with an action of $\mu_r$ with
 $\gcd(r,{\rm char}(\sO))=1$. Then the fixed ring $\sO^{\mu_r}$ is also a discrete
 valuation ring. We will assume that $\sO$ contains a field so that its completion
 $\hat{\sO}$ is a power series ring in one variable over the residue field.
 Note that $\mu_r$ must preserve the maximal ideal of $\sO$. If we further 
 assume that the action is generically free and inertial, i.e $\mu_r$ acts
 trivially on the residue field then if $s$ is a local parameter for $\sO$
 we can conclude that $t=s^r$ is a local parameter for $R=\sO^{\mu_r}$.
 
 We set $Y=\spec(R)$ and consider the root stack
 \[
 \fY = Y_{R, t, r}\rightarrow Y.
 \]
 The parameter $s$ induces a $\mu_r$-equivariant morphism 
 $$
 X \rightarrow \fY
 $$
 corresponding to the triple $(\sO,s,m)$ where $m$ is the canonical isomorphism 
 $\sO^r\rightarrow \sO$.
 We will show later (\ref{p:etale}) that this morphism is in fact \'{e}tale. Using the 2 out of 3 property for
 \'{e}tale maps we get that the natural morphism
 \[
 X\times\mu_r \rightarrow X\times_\fY X
 \]
 is \'{e}tale. To show that $[X/\mu_r]\cong \fY$ it suffices to show that this morphism is radicial 
 (universally injective) and surjective. In other words we need to show that it is a bijection on $K$-points
 for each field $K$.
 
 Given a pair of $K$-points $a$ and $b$ of $X$ that give a $K$-point of $X\times_Y X$ the fiber of
 \[
 X\times_\fY X\rightarrow X\times_Y X
 \]
 over this point consists of the space of isomorphisms between  $a^*(\sO,s,m)$ and $b^*(\sO,s,m)$
 in $\fY$. If the support of the $K$-points is the generic point of $\sO$ this is just a singleton
 and if the support is the closed point then the space is a bitorsor over $\mu_r$. 
 At any rate the morphism above is seen to be a an isomorphism. Hence in this case we have that
 $$
 [X/\mu_r]\cong \fY.
 $$
 \end{example}

 \begin{remark}\label{p:mu_torsor}
 A $\mu_r$-bundle $P$ on a scheme $Z$ is equivalent to the data of an invertible sheaf $\sK$ 
 and an isomorphism $\phi:\sK^r\rightarrow \sO_Z$. To construct $P$ explicitly consider the
 sheaf of algebras 
 $ {\rm Sym}^\bullet \sK^{-1} $. There is a distinguished global section $T\in\sK^{-r}$ given 
 by $(\phi\otimes 1_{\sK^{-r}}(1))$. Then 
 $$
 P = \spec\left({\rm Sym}^\bullet \sK^{-1}/(T-1)\right).
 $$
 \end{remark}

 \begin{remark}
 \label{r:isQuotient}
 Suppose that there is on $Y$ an invertible sheaf $\sN$ and an isomorphism $\sN^r\rightarrow \sL$. Then $Y_{\sL,s,r}$ is a global  quotient stack, see \cite[2.3.1 and 2.4.1]{C} and \cite[3.4]{borne}. We will need this below, so let's recall some of the details. The
 coherent sheaf
 \[
 \sA = \sO_Y\oplus \sN^{-1} \oplus \ldots \oplus \sN^{-(r-1)}
 \]
 can be given the structure of an $\sO_Y$-algebra via the composition 
 \[
 \sN^{-r} \stackrel{\sim}{\longrightarrow}\sL^{-1}\stackrel{s}{\longrightarrow}\sO_Z.
 \] There is an action of $\mu_r$ on this sheaf via the action of $\mu_r$ on $\sN^{-1}$
 given by scalar multiplication.
 Then $Y_{\sL,s,r}=[\spec(\sA)/\mu_r]$. We will need the explicit morphism 
 \[
 Y_{\sL,s,r}\rightarrow [\spec(\sA)/\mu_r]
 \]
 below so let's describe it. Consider a morphism $a:X\rightarrow Y$.
  A morphism $X\rightarrow Y_{\sL,s,r}$, lifting $a$, is a triple $(\sM,t,\phi)$.
 As per the previous remark the sheaf $\sM^{-1}\otimes \sN^{-1}$ gives a $\mu_r$-torsor.
 The torsor comes from the algebra 
 \[
 \sB = {\rm Sym}^\bullet \sM\otimes a^*\sN^{-1}/(T-1).
 \]
 To produce an $X$-point of $[\spec(\sA)/\mu_r]$ we need to describe a $\mu_r$-equivariant 
 map 
 \[
 a^*\sA \rightarrow \sB.
 \]
 This map comes from the section $t$ via :
 \[
 t\in \Hom(\sO,\sM)=\Hom(a^*\sN,\sM\otimes a^*\sN^{-1}).
 \]
 This construction generalizes in the obvious way to a finite list of invertible sheaves
 with section.
 \end{remark}
 
 \begin{assumption}
  \label{assumption}
  We will assume that $X$ and $Y$ are regular, separated, noetherian schemes over a field $k$.
  Let $G$ be a finite group with cardinality coprime to the characteristic of $k$. We will assume that
  $G$ acts admissibly and generically freely on $X$ with quotient $\phi:X\rightarrow Y$. Note that by \cite[Theorem 14.126]{gw}
  our hypothesis imply that the quotient map $X\rightarrow Y$ is flat.
   \end{assumption}
 
 Consider the map $\phi: X \rightarrow Y $ which is faithfully flat and finite. Recall that the set of points of $X$ where $\phi$ is ramified is called the branch locus. It has a natural closed subscheme structure defined by $\text{supp}(\Omega_{X/Y}).$ Because the conditions of the purity theorem  \cite[VI, Thm 6.8]{AK}  are satisfied in our situation this closed subscheme will give rise to an effective Cartier divisor which is called the branch divisor. We
 can write this divisor as
 \[
 R = \sum_{i=1}^n (r_i-1) \left( \sum_{g\in G} g^*D_i \right),
 \]
 where each $D_i$ is a prime divisor. As $G$ acts generically freely, passing
 to generic points of our regular variety produces a Galois extension with
 Galois group $G$. We can view the $D_i$ as points of the scheme $X$.
 The multiplicities $r_i$ are related to the inertia groups
 of $D_i$ via
 \[
 r_i = |I(D_i,G)|,
 \]
  see \cite[Ch. I.9]{Neuk}.
 
 We let $E_i$ be the image of $D_i$ under $\phi$.  It is called the ramification divisor. We form the root stack
 \[
 \fY = Y_{((E_1,r_1),\ldots, (E_n,r_n))}.
 \]    
 Note that we have assumed that the characteristic of our ground field is coprime to
 $G$ and hence to each $r_i$. It follows, via a local calculation along the ring extension
 $\sO_{X,D_i}/\sO_{Y,E_i}$ that we have
 $\phi^*(E_i)=r_i (\sum_{g\in G} g^* D_i)$. This allows us to lift $\phi$ to produce
 a diagram

     \begin{center}
 \begin{tikzpicture}
 \node (TL) at (0,0) {$X$};
 \node (BL) [below of =TL] {$Y$};
 \node (BR) [right of=BL] {$\fY$.};
 \draw [->] (TL) edge node[auto,left] {$\phi$} (BL) ;
 \draw [->] (TL) edge node[auto] {$\psi$} (BR);
 \draw [<-] (BL) edge node[auto] {$\pi$} (BR);
 \end{tikzpicture}
 \end{center}
The morphism
 $\psi$ is equivariant in the sense that precomposition with $g\in G$ produces a two-commuting diagram. This gives us a morphism
 \[
  [X/G]\rightarrow \fY
 \]
 that we would like to show is an isomorphism under our Assumption \ref{assumption} and the extra condition that the ramification divisor is normal crossing.
 
 In the proof below we will need to make use of 
 \begin{proposition}\label{p:abhyankar}
 (Abhyankar's lemma)
 Let $Y=\spec(A)$ be a regular local scheme and $D=\sum_{1\le i \le r} \div (f_i)$
 a divisor with normal crossings, so that the $f_i$ form part of a regular system
 of parameters for $Y$. Set $\bar{Y}={\rm Supp}(D)$ and let $U=Y\setminus \bar{Y}$.
 Consider $V\rightarrow U$ and \'{e}tale cover that is tamely ramified over $D$. If $y_i$ are the generic points of
 ${\rm supp}(\div(f_i))$ then $\sO_{Y,y_i}$ is a discrete valuation ring. If we let
 $K_i$ be its field of fractions then, as  $V$ ramifies tamely, we have that
 \[
 V|_{K_i}=\spec(\prod_{j\in J_i} L_{ji})
 \]
 where the $L_{ji}$ are finite separable extensions of $K_i$. We let $n_{ji}$ be the order
 of the inertia group of the Galois extension generated by $L_{ji}$ and let 
 \[
 n_i = \lcm_{j\in J_i} n_{ji},
 \]
 and set 
 \[
 A' = A[T_1,\ldots, T_r]/(T_1^{n_1}-f_1,\ldots, T_r^{n_r}-f_r)\qquad
 Y'=\spec(A').
 \]
 Then the \'{e}tale cover $V'=V\times_X X'$ of $U\times_X X'$ extends uniquely
 up to isomorphism to an \'{e}tale cover
 of $X'$.
 \end{proposition}
 
 \begin{proof}
 This is \cite[Expose XIII,  proposition 5.2]{SGA1}. The proof given shows how to 
 construct the extension of $V'$, we will need this below. The extension can be constructed
 as the normalization of $X'$ in the generic point of $V\times_X X'$.
 \end{proof}
 
\begin{proposition}
	\label{p:etale}
Under Assumption \ref{assumption}, suppose that $\phi: X\rightarrow Y$ is ramified along a simple normal crossings divisor $E$.	
The morphism $\psi:X\rightarrow \fY$ constructed above is \'etale.
\end{proposition}

\begin{proof}
\'Etale maps are local on the source so we can assume that $Y = \Spec(S),$ and all $E_i$ are trivial line bundles so that $s_i \in S$. Further, by shrinking $X$ we can
assume that the morphism $X \rightarrow \fY$ is defined be trivial bundles
on $X$.
Because the map $\phi$ is finite we can write $X = \Spec(T).$  Here $T$ and $S$ are local regular Noetherian $k$-algebras, $T$ is a finite $S$-module, $s_i$ is part of a regular system of parameters and there are elements $t_i \in T,$ such that $t_i^{r_i} = s_i$ . 

  We may check \'{e}taleness after a faithfully flat base extension of the base field
  and hence may assume that the ground field $k$ contains $r_i$-th roots of unity for all $1 \le i \le n. $    
  
  Using Remark \ref{r:isQuotient}, we can see that  the stack $\fY$ is isomorphic to the quotient stack 
  \[ 
[   \Spec(S')/\mu_{r_1}\times \dots \times \mu_{r_n} ], \]
 where $S' = S[y_1, \dots, y_n]/(y_1^{r_1}-s_1, \dots , y_n^{r_n}-s_n)$.

  We want to show that the map $\Spec(T) \rightarrow [\Spec(S')/\mu_{r_1}\times \dots \times \mu_{r_n} ]$ is \'etale. Denote by $T'$  the ring $T[x_1, \dots, x_n]/(x^{r_1}-1, \dots, x^{r_n}-1).$ Using Remark \ref{r:isQuotient} again we obtain
  a Cartesian diagram:
  
  \begin{center}
\begin{tikzpicture}
\matrix (m) [matrix of math nodes, row sep=4em, column sep=7em]
{\Spec(T') & \Spec(S') \\
\Spec(T) &  \big[ \Spec(S')/ \mu_{r_1}\times \dots \times \mu_{r_n} \big]    \\};
 \draw[->] (m-1-1) to node[left] {} (m-2-1);
 \draw[->] (m-1-2) to node[right]  {}(m-2-2);
 \draw[->] (m-1-1) to node {} (m-1-2);
 \draw[->] (m-2-1) to node {}(m-2-2);
  \end{tikzpicture}
  \end{center}
Because $\Spec(S')$ is a presentation of a quotient stack it is enough to show that the map $S' \rightarrow T'$ given by $y_i \mapsto t_i x_i$ is \'etale.

The morphism 
$S_{s_1\ldots s_n}\rightarrow T_{t_1\ldots t_n}$ is flat and unramified by assumption,
hence it is \'{e}tale. By Abhyankar's lemma, (Proposition \ref{p:abhyankar}), this morphism
extends after base change to an \'{e}tale cover of $S'$. By the proof of Abhyankar's lemma
it suffices to show that $T'$ is normal and the map $S'\rightarrow T'$ is integral. Both
of these facts are easily checked and the result follows.
\end{proof}

For a point $p\in Y$ we define
\[
I(p,Y) = \prod_{p\in {\rm supp}(E_i)} \mu_{r_i}.
\]

\begin{proposition}\label{p:torsor}
Let Assumption \ref{assumption} hold. Let $K$ be a field and consider the morphism of $K$-points
$$
\pi_K : X\times_\fY X (K) \rightarrow X\times_Y X (K).
$$
The fiber $\pi^{-1}_K(x_1,x_2)$ over a $K$-point $(x_1,x_2)$ is a bitorsor under
the inertia group
$I(\phi(x_1),Y)$.
\end{proposition}

\begin{proof} In what follows, we will use the shorthand $G^*$ when we mean $\sum_{g\in G} g^*$.
Recall that the morphism $\psi$ is defined by $(\sO(G^*E_i), s_{G^*E_i}, \alpha_i)$ where
$\alpha_i$ are isomorphisms coming from the fact that 
\[
r_iG^*E_i = r_i\phi^*(D_i).
\]
The fiber over $(x_1,x_2)$ is exactly the set of isomorphism from 
$x_1^* \sO(G^*E_i)$ to $x_2^*\sO(G^*E_i)$ as $i$ varies. As in Example \ref{e:dvr}, this
depends on whether the section $x_1^* s_{G^*E_i}$ vanishes or not. The vanishing condition
precisely depends on $\phi(x_1)$ and the result follows.
\end{proof}

The final ingredient we need to finish the proof is that under our assumptions the inertia group of $X$ is generated in codimension one. For that let us recall the following

\begin{proposition}
	\label{t:abhyankar}
	(Abhyankar's theorem \cite[Theorem2.3.2]{GrM}).
	
	Let $Y$ be a locally noetherian normal scheme, $D$ is a divisor with normal crossing, $\hat{Y} = \supp(D)$ and $U = Y \setminus \hat{Y}$. Assume that $X \rightarrow Y $ is a finite morphism, $G$ is a finite group operating on $X$, such that $X\mid U $ is a $G$-torsor. Then the following are equivalent:
	\begin{enumerate}[(i)]
\item $X$ is tamely ramified relative to $D$.
\item For every $y\in Y$ there exists an \'{e}tale neighborhood $Y'$ of $y$ in $Y$, and a scheme $S = \sO_Y[(T_i)_{i\in I'}]/((T_i^{r'_i} - f_i'))_{i\in I'} $, where $D_{Y'} = \sum_{i \in I'} D'_i$ and $\div(f'_i) = D'_i $, such that there is an isomorphism of couples:
$$(X', G) \simeq (G\times^H S, G), $$
where $X' = X\times_Y Y'$ and $H = \prod_{i\in I'}\mu_{r'_i}$. Let us recall that $G\times^H S$ is the quotient $(G\times S)/H$, where $H$ acts "by the formula": $h\cdot (g,s) = (gh^{-1}, hs) $.
	\end{enumerate}
\end{proposition}

Let us apply this fact to describe the inertia group. 

\begin{theorem}
	\label{t:inertia}
Under Assumption \ref{assumption}, suppose that $\phi: X\rightarrow Y$ is ramified along a simple normal crossings divisor. Then the inertia groups of $(X,G)$ are generated in codimension one.
\end{theorem}

\begin{proof}
	Firstly observe that condition (i) of Abhyankar's theorem (Proposition \ref{t:abhyankar}) is satisfied. Inertia is a local notion and also, clearly, the inertia group of $(S,H)$ is generated in codimension one. 
	
	There is an isomorphism of quotient stacks: $[(G\times^H S)/G] \cong [S/H] $. So inertia groups of $G\times^H S$ under the action of $G$ and of $S$ under the action of $H$ are isomorphic for the corresponding points. This finishes the proof. 
		\end{proof}

Finally we are ready to proof the main theorem of this section.

\begin{theorem}
	\label{main theorem}
  If Assumption \ref{assumption} is satisfied and if additionally the ramification divisor is a normal crossing divisor, then we have the isomorphism of stacks $[X/G] \cong \fY$.	
\end{theorem}

\begin{proof}
 To prove this all we need to show is that the map 
 $$ \chi: \quad   X \times G \rightarrow X\times_{\fY} X $$
     $$ \quad (x, g) \mapsto  (x,gx)                       $$ 
 is an isomorphism. 
 
   By Proposition \ref{p:etale}, the map $\psi: X \rightarrow \fY $ is \'etale, and so the map $X\times_\fY X \rightarrow X $ is \'etale as a pullback. Clearly two maps $X \times G \rightarrow X $ given by $(x, g) \mapsto x$ and $(x,g) \mapsto gx$ are \'etale and so the map $\chi$ must be \'etale.
   
    We are going to show that the map $$\chi(K): X(K) \times G \rightarrow X\times_\fY X (K) $$ is bijective for any field extension of the ground field $k \subset K.$ The points of the scheme on the left is a pair $(x, g),$ where $g \in G$ and $x: \Spec(K) \rightarrow X$ a $K$-point. 
    
   Consider the morphism $\Psi:X\times G\rightarrow X\times_Y X$. This morphism is surjective
   as we have a geometric quotient, see \cite[Definition 0.4]{git}. Consider a $K$-point $(x_1,x_2)\in X\times_Y X(K)$. Using the properties of geometric quotients we have
   that 
   $x_2= g x_1$ for some $g\in G$. Using this we see the fiber $\Psi^{-1}(x_1,x_2)$
   is a torsor over the inertia group $I({\rm supp}(x_1),G)$. By Theorem \ref{t:inertia} our inertia groups are generated in codimension one, so we see that we have
   an identification 
   \[
   I({\rm supp}(x_1),G) = \mu_{r_{i_1}}\times\ldots \times \mu_{r_{i_l}}
   \]
   as in Proposition \ref{p:torsor}. 
   It follows that the morphism $\chi$ is \'{e}tale and universally injective (radical). This
   implies that it is an open immersion. As it is also surjective, it is an isomorphism and
   the result follows.
\end{proof}

 \section{An application of root stacks to the equivariant $K$-theory of schemes}
 \label{s:equi}

   As an application of the theorems proved in the sections 3 and 4, we can formulate a result about equivariant $K$-theory.
   
   \begin{theorem}
   	\label{K-theorem}
    Let $X$ be a regular, separated, noetherian scheme over the field $k$ with a generically free admissible action of a finite group $G,$ such that the order of $G$ is coprime to the characteristic of $k$. Let's denote by $X/G = Y $ and assume that all the conditions of  Assumption \ref{assumption} are satisfied. Also assume that $X \rightarrow Y$ is ramified along a simple normal crossing divisor $E$. Then there is an isomorphism of groups:
        $$K^\bullet_G(X) \cong K^\bullet (Y) \oplus (\oplus_{i=1}^n ( \oplus_{T \in S(i) } G^\bullet (\bigcap_{l \in T} E_l)^{\oplus \prod_{l \in T}(r_l - 1)}) ), $$
where $r_l$ are orders of inertia groups (see Section 4 for notation), and $S(i) = \{T \subset \{1, \dots, n\}\mid |T|=i \}$.    
   \end{theorem}
 
 \begin{proof}
 
  By  assumption $X$ is a regular scheme and the group $G$ is finite so for any $G$-equivariant sheaf we can always construct an equivariant locally free resolution by averaging  the usual locally free resolution. This simple argument shows that the equivariant $K$-theory of $X$ is the same as the equivariant $G$-theory.
    
     The category of $G$-equivariant sheaves on $X$ is equivalent to the category of sheaves on the quotient stack $[X/G]$ so we can see that 
     $$K_G(X) \cong G([X/G]). $$
In  Theorem \ref{main theorem} we proved that under our assumptions that there is an isomorphism of stacks $[X/G] \cong \fY, $ so we have an isomorphism of their $G$-theories:
 
          $$G([X/G]) \cong G( \fY) . $$
          
   Finally the application of Theorem \ref{th:G-root} gives the desired formula.      
 \end{proof}
 
 Let us give some examples. 
 
 \begin{example}
 	Let's consider $\mathbb{A}^1$ over a field $k$ with an action of $\mu_3$ (it acts by multiplication). Assume that $\text{char}(k) \neq 3$. Then $\mathbb{A}^1 / \mu_3 \cong \mathbb{A}^1$ and ramification divisor is $\div(0)$. The inertia group is $\mu_3$.  So by Theorem \ref{K-theorem}:
 	$$K^\bullet_{\mu_3}(\mathbb{A}^1 ) \cong K^\bullet (\mathbb{A}^1) \oplus K^\bullet (k) \oplus K^\bullet (k) \cong  K^\bullet (k)^{\oplus 3 }.$$
 	
 \end{example}
 
 \begin{example}
 	This example was inspired by the paper \cite{AO}. The Burniat surface $X$ with $K_X^2 = 6$ is a Galois $G := C_2 \times C_2$-cover of $\text{Bl}_3\mathbb{P}^2$ (a del Pezzo surface of degree 6). Let's assume that the ground field $k$ is algebraically closed and $\text{char}(k) \neq 2$. The ramification divisor is given in Figure 1 \textit{loc. cit}: it is denoted by  $A_l, \ B_l, \ C_l$, where $0 \le l \le 4$. The inertia group of each component is $C_2$, the inertia group of an intersection point of any two components is $G$. The intersection of three components is empty. Also $A_l \cong B_l \cong C_l \cong \mathbb{P}^1 $, for all $l = 0, \ldots, 3 $. 
 	
 	Applying Theorem \ref{K-theorem} one gets:
 	
 	$$ K^\bullet_G(X) \cong K^\bullet (\text{Bl}_3\mathbb{P}^2) \oplus (\oplus_{i=1}^{2} Z_i^\bullet ) , $$
 	$$Z_1^\bullet = K^\bullet (\mathbb{P}^1)^{\oplus 12}  ,$$
 	$$ Z_2^\bullet =  K^\bullet(k)^{\oplus 30} .$$

 \end{example}



\begin{thebibliography}{ABCD}

\bibitem[AGV]{AGV}  D. Abramovich, T. Graber, A. Vistoli , \textit{Gromov-Witten theory of Deligne-Mumford stacks},  American Journal of Mathematics, \textbf{130} (2008), No 5,  1337-1398

\bibitem[AK]{AK} A. Altman and S. Kleiman \textit{Introduction to Grothendieck Duality Theory}, 
LNM \textbf{146}  (1970), Springer-Verlag

\bibitem[AO]{AO} V. Alexeev and D. Orlov \textit{Derived categories of Burniat surfaces and exceptional collections}, Math. Ann. \textbf{357} (2013), 743-759

\bibitem[B]{borne} N. Borne, \textit{Fibr\'{e}s paraboliques et champ des racines},
Int. Mat. Res. Not. 2007, (2007)

\bibitem[BV]{BV} N. Borne and A. Vistoli \textit{Parabolic sheaves on logarithmic schemes}, Adv. Math. \textbf{60} (2012), No. 231, 1327-1363.

\bibitem[Bou]{Bourbaki} N. Bourbaki, \textit{Groupes et alg\`{e}bres de Lie}, Chapitres 4 \`{a} 6, Hermann,‎ (1968).

\bibitem[C]{C} C. Cadman \textit{Using stacks to impose tangency conditions on curves}, Amer. J. Math. \textbf{129} (2007), No. 2, 405--427.


\bibitem[EL]{duke} G. Ellingsrud and K. Lonsted, \textit{Equivariant $K$-theory for curves}, Duke Math. J, \textbf{51}, (1984)


\bibitem[FKM]{git} J. Fogarty and F. Kirwan and D. Mumford \textit{Geometric Invariant theory}, 3rd edition, Ergebnisse der Mathematik und ihrer Grenzgebiete \textbf{34}, 
Springer-Verlag (1994) 

\bibitem[GS]{GS} A. Geraschenko and M. Satriano, \textit{A ``bottom up" characterization of smooth Deligne-Mumford stacks}, Int.  Math. Res. Not., Vol. \textbf{2017}, Issue 21, (2017), pp. 6469-6483

\bibitem[GW]{gw} U. G{\"o}rtz  and T. Wedhorn, \textit{Algebraic geometry I},
Advanced lectures in mathematics, (2010)


\bibitem[SGA1]{SGA1} A. Grothendieck,  \textit{SGA I}, LNM \textbf{224} (1971), Springer-Verlag

\bibitem[GM]{GrM} A. Grothendieck, J. P. Murre  \textit{The tame fundamental group of a formal neighborhood of a divisor with normal crossings on a scheme}, LNM \textbf{208} (1971), Springer-Verlag

\bibitem[H]{hagihara} K. Hagihara, \textit{Structure theorem of Kummer \'{e}tale K-group},
 $K$-theory, \textbf{29}, (2003)
	

\bibitem[J]{Joshua} R. Joshua,  \textit{$K$-Theory and Absolute Cohomology for algebraic stacks}, preprint, http://www.math.uiuc.edu/K-theory/0732/


 
   
\bibitem[N]{Neuk} J. Neukirch    \textit{Algebraic number theory},
   Grundlehren der Mathematischen Wissenschaften,
   322, (1999)
   
\bibitem[Ni]{niziol} W. Niziol, \textit{$K$-theory of log-schemes I}, Doc. Math, \textbf{13}, (2008)
   
   \bibitem[MS]{MS} V. B. Mehta, C. S. Seshadri  \textit{Moduli of vector bundles on curves with parabolic structures}, Math. Ann.    \textbf{248} (1980) 205-239 
   
   
  \bibitem[O]{O} M. Olsson \textit{(Log) twisted curves}, Compositio Math.
  \textbf{143} (2007) 476–494
   
\bibitem[Q]{Q} D. Quillen \textit{Higher algebraic K-theory: I}, LNM \textbf{341} (1973), Springer-Verlag 	

\bibitem[Se]{segal} G. Segal  \textit{Equivariant $K$-theory}, Publ. Math. IHES, \textbf{34} (1968), pp. 129-151.      

\bibitem[S]{swan} R. Swan, \textit{Algebraic K-theory}, LNM \textbf{76} (1968), Springer-Verlag

\bibitem[TV]{tv} M. Talpo, A. Vistoli, \textit{Infinite root stacks and quasi-coherent sheaves on logarithmic schemes}, Proc. Lond. Math. Soc. (3),
116, (2018)

\bibitem[VV]{vv} G. Vezzosi, A. Vistoli, \textit{Higher algebraic K-theory of group actions with finite stabilizers}, Duke Math. J., \textbf{113} , 1 (2002)

\bibitem[V]{vistoli} A. Vistoli, \textit{Higher equivariant $K$-theory for finite group actions}, Duke Math. J., \textbf{63} (1991)




\end{thebibliography}
\end{document}